\documentclass{article}
\usepackage[utf8]{inputenc}
\usepackage[legalpaper, margin=1in]{geometry}

\usepackage{times}
\usepackage{epsfig}
\usepackage{graphicx}
\usepackage{amsmath}
\usepackage{amssymb}
\usepackage{tabularx} 
\usepackage{mathtools} 
\usepackage{amsfonts}
\usepackage{amsthm}
\usepackage{booktabs}

\newcommand{\MB}[1]{{\color{purple}{\textbf{MB:} #1 }}}

\newtheorem{innercustomgeneric}{\customgenericname}
\providecommand{\customgenericname}{}
\newcommand{\newcustomtheorem}[2]{%
  \newenvironment{#1}[1]
  {%
   \renewcommand\customgenericname{#2}%
   \renewcommand\theinnercustomgeneric{##1}%
   \innercustomgeneric
  }
  {\endinnercustomgeneric}
}

\newcustomtheorem{customtheorem}{Theorem}
\newcustomtheorem{customproposition}{Proposition}
\newcustomtheorem{customlemma}{Lemma}

\usepackage[pagebackref=true,breaklinks=true,letterpaper=true,colorlinks,bookmarks=false]{hyperref}

\usepackage{tabularx} 
\usepackage{amsmath}
\usepackage{graphicx} 
\usepackage{amsmath}
\usepackage{mathtools} 
\usepackage{amsfonts}
\usepackage{amssymb}
\usepackage{amsthm}
\usepackage{enumitem}
\usepackage{dsfont}
\usepackage{tikz-cd} 
\usepackage{comment}
\usepackage[utf8]{inputenc}
\usepackage{changepage}
\usepackage[format=plain, font=footnotesize]{caption}
\usepackage[sort]{cite}
\usepackage{ulem}
\usepackage[obeyFinal]{todonotes}

\hypersetup{
	colorlinks=true,       
	linkcolor=teal,        
	citecolor=orange,        
	filecolor=magenta,     
	urlcolor=blue         
}
\usepackage{blindtext}

\newcommand{\PP}{\mathbb{P}}
\newcommand{\RR}{\mathbb{R}}

\newcommand{\CC}{\mathbb{C}}
\newcommand{\NN}{\mathbb{N}}

\newcommand{\QQ}{\mathbb{Q}}

\newcommand{\Set}[1]{\left\{#1\right\}}
\newcommand{\p}[1]{\mathbb{P}^{#1}}

\DeclareMathOperator{\PGL}{PGL}

\usepackage[capitalise, noabbrev, nameinlink]{cleveref}

\theoremstyle{plain}

\newtheorem{theorem}{Theorem}[section] 
\newtheorem{proposition}[theorem]{Proposition}
\newtheorem{lemma}[theorem]{Lemma}

\newtheorem{example}[theorem]{Example}
\newtheorem{remark}[theorem]{Remark}

\begin{document}
\title{Compatibility of Fundamental Matrices for Complete Viewing Graphs}

\author{Martin Bråtelund\\
University of Oslo\\
Moltke Moes vei 35, 0851 Oslo, Norway\\
{\tt\small mabraate@math.uio.no}
\and
Felix Rydell\\
KTH Royal Institute of Technology\\
Lindstedtsvägen 25, Stockholm, Sweden\\
{\tt\small felixry@kth.se}
}

\maketitle

\begin{abstract} 

This paper studies the problem of recovering cameras from a set of fundamental matrices. A set of fundamental matrices is said to be compatible if a set of cameras exists for which they are the fundamental matrices. We focus on the complete graph, where fundamental matrices for each pair of cameras are given. Previous work has established necessary and sufficient conditions for compatibility as rank and eigenvalue conditions on the $n$-view fundamental matrix obtained by concatenating the individual fundamental matrices. In this work, we show that the eigenvalue condition is  redundant in the generic and collinear cases. We provide explicit homogeneous polynomials that describe necessary and sufficient conditions for compatibility in terms of the fundamental matrices and their epipoles. In this direction, we find that quadruple-wise compatibility is enough to ensure global compatibility for any number of cameras. We demonstrate that for four cameras, compatibility is generically described by triple-wise conditions and one additional equation involving all fundamental matrices.

\end{abstract}


\section*{Introduction} 

The problem of finding camera matrices that correspond to a given set of fundamental matrices is crucial in 3D reconstructions from 2D images. Typically, multiview structure-from-motion pipelines start by estimating fundamental matrices from point correspondences, with early methods for such estimations dating back to the 1990s and new methods still being developed today \cite{torr1995motion,xu1996epipolar,torr1997development,ranftl2018deep}. However, these methods usually only estimate a subset of all possible fundamental matrices between cameras. To describe this incomplete set of fundamental matrices, viewing graphs are often used \cite{levi2003viewing}.


In this paper, we focus on understanding the conditions under which a reconstruction of $n$ cameras can be obtained given complete knowledge of ${n\choose 2}$ fundamental matrices, but we also give a result for general graphs at the end. Here, a \emph{camera} refers to a full-rank $3\times 4$ matrix, and the \emph{fundamental matrix} of two cameras $P_1$ and $P_2$ with distinct kernels is a $3\times 3$ rank-2 matrix that encodes all point correspondences between them. For any given rank-2 $3\times 3$ matrix $F^{12}$, there exists a pair of cameras $P_1$ and $P_2$ for which $F^{12}$ is the fundamental matrix, this pair is unique up to global projective transformation. However, for a set of ${n\choose 2}$ rank-2 $3\times 3$ matrices $F^{ij}$, where $n>2$, it is not guaranteed that there exist cameras $P_1,\ldots,P_n$ such that $F^{ij}$ is the fundamental matrix of $P_i$ and $P_j$ for each $i,j$. Following the notation of \cite{Hartley2004} we say that the set ${F^{ij}}$ is \emph{compatible} if such cameras do exist. Note that some recent literature uses the term \emph{consistent} instead \cite{kasten2019gpsfm}.

Finding necessary and sufficient conditions for compatibility of fundamental matrices has practical applications as well as theoretical ones. \cite{kasten2019gpsfm} proposes an algorithm for projective structure-from-motion that employs their necessary and sufficient condition for compatibility. The algorithm is designed to handle collections of measured fundamental matrices, both complete and partial, and aims to find camera matrices that minimize a global algebraic error for the given set of matrices. As for theoretical purposes, \cite{braatelund2021critical,braatelund2021threeviews,HK} uses necessary and sufficient conditions for compatibility to give a classification of critical configurations. 


In the case of $n=3$, a classical result \cite[Section 15.4]{Hartley2004} provides triple-wise constraints on $F^{12},F^{13},F^{23}$ in terms of the fundamental matrices and their epipoles, where the $i$-th epipole in the $j$-th image is defined as $e_j^i:=\ker F^{ij}$. For non-collinear cameras, \cite[Theorem 1]{kasten2019gpsfm} provides necessary and sufficient conditions for compatibility for any $n$. These conditions rely on the eigenvalues and rank of the $n$-view fundamental matrix, which is obtained by stacking all fundamental matrices into a $3n\times 3n$ matrix. In the follow-up work, \cite[Theorem 2]{geifman2020averaging} arrives at a similar condition in the collinear case. Both methods rely on fixing a correct scaling of each matrix and are therefore not projectively well-defined, nor are the conditions expressed in terms of the fundamental matrices and their epipoles, as in the $n=3$ case.

The contributions of this paper include giving explicit homogeneous polynomials that provide necessary and sufficient conditions for the compatibility of fundamental matrices in the case of complete graphs. The paper is structure as follows. In \Cref{s: Action}, we introduce the fundamental action, a key tool in simplifying the problem of finding compatibility conditions. In \Cref{s: Kn}, for the case of $n=4$, we establish that a set of six fundamental matrices admits a reconstruction of camera matrices with linearly independent centers only if the triple-wise constraints and one additional polynomial equation involving all six fundamental matrices and their epipoles are satisfied. We also demonstrate, using the computer algebra system \texttt{Macaulay2} \cite{M2}, that the eigenvalue conditions from \cite{kasten2019gpsfm, geifman2020averaging} are superfluous in the generic case and in the case where all epipoles in each image coincide. \Cref{s: Can} presents a necessary and sufficient condition for compatibility for any viewing graph via a cycle condition, similar to cycle-based formulations of parallel rigidity that appear in the calibrated case. Finally, in \Cref{s: Image}, we discuss the image of the fundamental map and prove a first result on this topic. 

We approach compatibility of fundamental matrices from an algebraic point of view, i.e., we aim to describe constraints through algebraic equations and polynomial equations using techniques and software from applied algebraic geometry. This approach to questions in computer vision has a long standing tradition \cite{heyden1997algebraic,trager2015joint,agarwal2019ideals,duff2019plmp,kileel2022snapshot}.

\subsection*{Related work} 

\paragraph{History.} The problem of determining whether a set of fundamental matrices is compatible has a curious history. \cite{Hartley2004} provided a necessary and sufficient triple-wise condition for the compatibility of three fundamental matrices $F^{12},F^{13}$ and $F^{23}$ arising from three cameras with non-collinear centers. In 2007, the paper \cite[Theorem 2.2]{HK} claimed that this condition was sufficient for compatibility even in the case of cameras with collinear centers, a claim that we show to be false in \Cref{ex:counterexample1}. During the next decade, few advances were made in understanding compatibility. Over time, a belief seemed to develop that triple-wise compatibility was enough to ensure global compatibility. In fact, articles such as \cite[Section 2.1]{falseclaim2} claimed this to be true, based on a faulty proof provided in \cite{rudi2011linear}. In 2018, \cite[Section 3.3]{trager2018solvability} pointed out that the proof in \cite{rudi2011linear} fails in some cases, but still agreed that the result holds for complete graphs. \Cref{ex:counterexample2} shows that this is not the case by providing a counterexample.

\paragraph{Essential matrices.} In the context of \emph{uncalibrated cameras}, which are defined as full-rank $3\times 4$ matrices, this work, as well as \cite{kasten2019gpsfm}, provide necessary and sufficient conditions for compatibility of fundamental matrices. However, camera matrices are often assumed to be \emph{calibrated}, represented in the form of $[R|t]$ for a rotation matrix $R$ and a translation vector $t$. The corresponding fundamental matrices are called \emph{essential matrices}. In \cite{kasten2019algebraic}, the authors build upon their previous work and provide a necessary and sufficient condition for compatibility of essential matrices, in terms of the $n$-view essential matrix obtained by stacking all essential matrices into a larger matrix. This condition is then used to recover a consistent set of essential matrices, given a partial set of measured essential matrices. In \cite{martyushev2020necessary}, Martyushev provides a necessary and sufficient condition for compatibility of three essential matrices.


\paragraph{Solvability.} There has been extensive research on the topic of solvability of viewing graphs in computer vision, as evidenced by various studies such as \cite{levi2003viewing,rudi2011linear,nardi2011augmented,trager2015joint,trager2018solvability,arrigoni2021viewing,arrigoni2022revisiting}. A viewing graph is considered \emph{solvable} if, given a generic set of cameras, their fundamental matrices have a unique solution in terms of cameras up to global projective transformation. Recently, \cite{arrigoni2021viewing} proposed a new formulation of solvability and developed an effective algorithm for testing it.

The primary distinction between solvability and compatibility lies in the fact that, in the latter, the existence of cameras that correspond to a set of fundamental matrices is not assumed to exist. Moreover, compatibility has mostly been studied for graphs where each possible fundamental matrix is given, whereas papers on solvability study viewing graphs without such restrictions.

Furthermore, solvability has been investigated in the case of calibrated cameras, where it is known that the solvable graphs are precisely those that are parallel rigid \cite{ozyesil2015robust,sattler2016efficient}.

\begin{small}
\paragraph{\begin{small}
    Acknowledgements.
\end{small}} The authors would like to thank Kathlén Kohn, Kristian Ranestad, Timothy Duff, and Paul Breiding for helpful discussions, and Erin Connelly for pointing out a sign error in one of our proofs. Martin Bråtelund was supported by the Norwegian National Security Authority. Felix Rydell was supported by the Knut and Alice Wallenberg Foundation within their WASP (Wallenberg AI, Autonomous Systems and Software Program) AI/Math initiative.
\end{small}

\setcounter{section}{0}
\section{Preliminaries}\label{s: Pre} In this section we recall established notation and results, as well as concepts of algebraic geomety in \Cref{ss: alg geo}. We work over the real numbers, although all results in this paper either directly hold in the complex case or can be reformulated to do so. Where slight adjustments have to be made over the complex numbers, we make a remark.  

Let $\RR^n$ denote the set of real vectors with $n$ coordinates, we call this \emph{affine space}. Let $\PP^{n-1}$ denote its projectivization. We write $\RR^{n\times m}$ to denote the set of real $n\times m$ matrices, and we write $\PP^{n\times m}$ to denote the set of real projective $n\times m$ matrices.

We define a rational map (this notion is formally defined in \Cref{ss: alg geo})
\begin{equation}
    \psi : \PP^{3\times 4}\times \PP^{3\times 4}\dashrightarrow \PP^{3\times 3},
\end{equation}
as follows. Given a pair of $3\times 4$ matrices ${P}_1$ and ${P}_2$ (defined up to scale), let $\textbf{x}$ and $\textbf{y}$ be two $3\times1$ vectors. The determinant
\begin{align}\label{eq: det}
\det\begin{bmatrix}
{P}_1&\textbf{x}&0\\
{P}_2&0&\textbf{y}
\end{bmatrix}
\end{align}
is a bilinear polynomial with in $\textbf{x}$ and $\textbf{y}$, meaning there is a matrix ${F}^{12}$ (defined up to scale) such that (\ref{eq: det}) can be written as $\textbf{x}^T{F}^{12}\textbf{y}$. We define $\psi({P}_1,{P}_2)$ to be this $3\times 3$ matrix. This map is undefined, i.e. $\psi(P_1,P_2)=0$, precisely when $\ker P_1\cap \ker P_2\neq \{0\}$. 

We refer to rank-2 $3\times 3$ matrices as \textit{fundamental matrices} (either in $\RR^{3\times 3}$ or $\PP^{3\times 3}$) and we refer to rank-3 $3\times 4$ matrices as \emph{cameras} (either in $\RR^{3\times 4}$ or $\PP^{3\times 4}$).  The \emph{center} of a camera $P$ is its kernel $\ker P$. 
Before we list a set of well-known results, partly found in \cite[Section 9]{Hartley2004}, we recall that $\mathrm{GL}_n$ denotes the set of invertible $n\times n$ matrices and that $\mathrm{PGL}_n$ is its projectivization.

\begin{proposition}\label{prop:properties_of_the_fundamental_matrix} $ $

\begin{enumerate} 
      \item $\psi(P_1,P_2)$ is of rank at most 2, and it attains this rank if $P_1,P_2$ are cameras with distinct centers;
    \item for any fundamental matrix $F^{12}$, there exist two cameras $P_1,P_2$ such that $F^{12}$ is their fundamental matrix. All other cameras $C_1,C_2$ with fundamental matrix $F^{12}$ satisfy $C_1=P_1H,C_2=P_2H$ for some $H\in \mathrm{PGL}_4$;
    \item $\psi({P}_2,{P}_1)=\psi({P}_1,{P}_2)^T$;
    \item if $F^{12}$ is the fundamental matrix of $P_1,P_2$, then $\ker F^{12}=P_2 \ker(P_1)$;
    \item for cameras $P_1,P_2$, we have $F^{12}=\psi(P_1,P_2)$ if and only if $P_1^TF^{12}P_2$ is a skew-symmetric matrix.
\end{enumerate}
\end{proposition}

We say that a set of fundamental matrices $\{F^{ij}\}$ is \emph{compatible} if there are cameras $P_1,\ldots,P_n$ such that $F^{ij}=\psi(P_i,P_j)$. The cameras $P_1,\ldots,P_n$ are called a \emph{solution} to $F^{ij}$. We mostly focus on complete viewing graphs, i.e. when $\Set{F^{ij}}$ contains all $n\choose 2$ fundamental matrices for $n$ indices. Still, in \Cref{s: Can}, we provide a result that holds not only in this setting, but for any viewing graph.

We define the $i$-th \emph{epipole} $e_j^i$ in the $j$-th image to be an affine representative of $\ker F^{ij}$. By \Cref{prop:properties_of_the_fundamental_matrix} 4., $e_j^i$ is the image of the $i$-th camera center taken by the $j$-th camera.

\begin{lemma}[{\hspace{1sp}\cite[Section 15.4]{Hartley2004}}]
\label{lem:triple_has_unique_solution}
    Let $\Set{F^{12},F^{13},F^{23}}$ be compatible. There is a unique solution if and only if the two epipoles in each image are distinct.
\end{lemma}

Although fundamental matrices and epipoles are only defined up to scale, i.e. as elements in projective space, we always assume for convenience that we are given affine representatives of them and that the representatives of fundamental matrices satisfy $(F^{ij})^T=F^{ji}$, unless otherwise is specified.

Given a fixed set of fundamental matrices $F^{ij}$, we point out that there is a rather simple method of finding possible solutions in terms of cameras by first using $F^{12}$ to recover $P_1,P_2$ and then using \Cref{lem:triple_has_unique_solution} with matrices $\Set{F^{12},F^{1i},F^{2i}}$ to recover the remaining $P_i$ (a detailed algorithm can be found in \cite[Section 6.1]{HK}). Finding explicit equations in terms of the fundamental matrices and epipoles for compatability is however more difficult, and is the subject of this paper.


\subsection{Methods of algebraic geometry} \label{s: Prelim App} \label{ss: alg geo}

For this paper, it is helpful to understand saturation and elimination of \emph{ideals}. We refer the reader to \cite{Gathmann} for the basics on algebraic geometry and \cite{cox1994ideals} for a detailed study of these topics.  Consider a field $k$ and its polynomial ring $k[x]=k[x_1,\ldots,x_m]$; the set of all polynomials with coefficients in $k$. That $k[x]$ is a \emph{ring} means that addition and multiplication of polynomials satisfy a certain set of axioms that we don't list here. An \emph{ideal} $I$ of a ring $R$ is an additive subgroup that is closed under multiplication of elements in $R$. 

Let $f_1,\ldots,f_s\in k[x]$ be polynomials. They generate an ideal of $k[x]$ as follows:
\begin{align}
    \langle f_1,\ldots,f_s\rangle:=\big\{\sum g_if_i: g_i\in k[x]\big\}\subseteq k[x].
\end{align}
From the geometric point of view, an ideal in a polynomial ring defines a \emph{variety} $\mathcal V$ as the zero set of all polynomials in the ideal. In other words,
\begin{align}
    \mathcal V(I):=\big\{ x\in k^m: f(x)=0\;\forall f\in I\big\}.
\end{align}
The Zariski closure $\overline{U}$ of a set $U\subseteq k^m$ is the smallest variety $X$ that contains $U$. 

The goal of saturation is to remove unwanted components from a variety. Let $I,J$ be ideals. The \emph{saturation} of $I$ with respect to $J$ is
\begin{align}\begin{aligned}
    I:J^\infty:=\big\{f\in k[x]: &\forall g\in J,\; \exists N\in \NN 
    &\textnormal{such that } fg^N\in I\big\}.   
\end{aligned}
\end{align}
It follows from definition that $I:J^\infty=I:(I+J)^\infty$, where $I+J=\{g+f:g\in I,f\in I\}$, and therefore we may assume without restriction that $I\subseteq J$.

\begin{theorem}[{\hspace{1sp}\cite[p. 203]{cox1994ideals}}]\label{thm: sat} Let $\mathcal{V}(J),\mathcal{V}(I)$ be two varieties over any field $k$. Then
\begin{align}
    \overline{\mathcal{V}(I)\setminus \mathcal{V}(J)}\subseteq \mathcal{V}(I:J^\infty).
\end{align}
\end{theorem}

The \emph{elimination} of variables $x_1,\ldots,x_l$ from an ideal $I\subseteq k[x]$ is the intersection
\begin{align}\label{eq: Icapk}
    I\cap k[x_{l+1},\ldots,x_m].
\end{align}
Given $(x_1,\ldots,x_n)\in \mathcal{V}(I)$, we have that $(x_{l+1},\ldots,x_n)\in \mathcal{V}(I\cap k[x_{l+1},\ldots,x_m])$, because any $f$ in \Cref{eq: Icapk} also lies in $I$. In this way, elimination of variables gives us conditions on the projection of $\mathcal{V}(I)$ away from the first $l$ coordinates.

In \Cref{s: Kn}, we use the symbolic programming language \texttt{Macaulay2} \cite{M2} to symbolically saturate ideals and eliminate variables in the ring $\QQ[x]$. In our study, all polynomials have rational coefficients, i.e. are elements of $\QQ[x]$. However, our varieties lie in real space. For saturation and elimination, it may matter in which ring the operations are performed in. In \texttt{Macaulay2} all such operations happen inside $\QQ[x]$, and we therefore prove the following lemma for clarity.

\begin{lemma}\label{le: from Q to R} Let $I,J$ be ideals in $\RR[x]$ generated by elements of $\QQ[x]$. Write $I_Q,J_Q\subseteq \QQ[x]$ for the ideals defined as the intersections $I\cap \QQ[x], J\cap \QQ[x],$ respectively. If $y\in \RR^m$ lies in $\mathcal V(I)\setminus \mathcal V(J)$, then $f(y)=0$ for every $f$ in the saturation $I_Q:J_Q^\infty$ performed inside the ring $\QQ[x]$.
\end{lemma}

Hence saturation in $\mathbb{Q}[x]$ tell us something also for the real numbers. 
The statement and proof works the same if $\RR[x]$ is replaced by $\CC[x]$. 

\begin{proof} By \Cref{thm: sat}, $y\in \mathcal V(I:J^\infty)$. It suffices to show that $I_Q:J_Q^\infty\subseteq I:J^\infty$, since then we have $\mathcal V(I:J^\infty)\subseteq \mathcal V(I_Q:J_Q^\infty)$ over the real numbers. Let $f\in I_Q:J_Q^\infty$
. Then $f\in \QQ[x]$ and for every $g\in J_Q$, there is an $N$ such that $fg^N\in I_Q$. Let $g_1,\ldots,g_k\in \QQ[x]$ generate $J$ and $J_Q$. Let $N_i$ denote an integer such that $fg_i^{N_i}\in I_Q$. Now take any $g\in J$. We can write $g=\sum_{i=1}^k h_ig_i$ for some $h_i\in \RR[x]$. There is an integer $N$ depending on $k$ and $N_i$ such that each term of $g^N$ is divisible by some $g_i^{N_i}$ and $fg^N\in I_Q$. For such $N$, we can write $g^N=\sum_{i=1}^k h_i'g_i^{N_i}$ for some $h_i'\in \RR[x]$. Then, since $fg_i^{N_i}\in I_Q$, we must have that $fg^N\in I$. This shows that inclusion $I_Q:J_Q^\infty\subseteq I:J^\infty$ and we are done. 
\end{proof}


In the main body of the text, the term rational map was used, which we now define. A variety $\mathcal V$ is called \emph{irreducible} if it cannot be written as a union of two proper varieties, meaning that for two subvarieties $X,Y$ of $\mathcal{V}$, the equality $\mathcal V= X\cup Y$ implies $\mathcal V=X$ or $\mathcal V=Y$. A \emph{rational map} $f$ between projective varieties $X$ and $Y$, with $X$ irreducible, is defined on a \emph{Zariski open} set of $X$, which is a set that can be written $X\setminus Y$ for a proper subvariety $Y\subseteq X$. A rational map between $X$ and $Y$ is written
\begin{align}
    f: X\dashrightarrow Y.
\end{align}

\section{The Fundamental Action}\label{s: Action}


In this section, we formally introduce the fundamental action, a key tool in simplifying the problem of finding compatibility conditions. 
$\mathrm{GL}_3^n$ (or equivalently $\mathrm{PGL}_3^n$) acts on a set of fundamental matrices $\Set{F^{ij}}$ by
\begin{align}
    \Set{F^{ij}}\mapsto \Set{H_i^TF^{ij}H_j}.
\end{align}
We call this the \emph{fundamental action} of $\mathrm{GL}_3^n$. The main appeal of this action is that we can use it to simplify a set of fundamental matrices, without affecting compatibility.

\begin{proposition} 
\label{prop:fundamental_action_preserves_compibility} Let $\Set{F^{ij}}$ be a set of fundamental matrices. Let $P_i$ be a solution to $\Set{F^{ij}}$. For any $(H_1,\ldots,H_n,H)\in \mathrm{PGL}_3^n\times \mathrm{PGL}_4$, we have, 
\begin{align}\label{eq: fund}
\psi(H_i^{-1}P_iH,H_j^{-1}P_jH)=H_i^T\psi(P_i,P_j)H_j.
\end{align} 
In particular, $\Set{F^{ij}}$ is compatible if and only if $\Set{G^{ij}}$ is compatible, where $G^{ij}:=H_i^TF^{ij}H_j$. 
\end{proposition}

\begin{proof} It is a standard fact that the action of $H\in \mathrm{PGL}_4$ in \Cref{eq: fund} does not change the fundamental matrix, so we may set $H=I$. Consider the following equality up to scaling,
\begin{align}
    \det \begin{bmatrix} H_i^{-1}P_i & x_i & 0 \\
H_j^{-1}P_j & 0 & x_j\end{bmatrix}=\det \begin{bmatrix}P_i & H_ix_i & 0 \\
P_j & 0 & H_jx_j\end{bmatrix}.
\end{align}
Writing these expressions in terms of fundamental matrices, we get exactly \Cref{eq: fund}. \end{proof}


The fundamental action gives rise to an equivalence relation. 
For compatible fundamental matrices, the equivalence classes turn out to be the equivalence classes of $n$ points in $\p3$ under $\PGL_4$.

\begin{proposition} \label{prop:equivalence} Let $\Set{F^{ij}}$ and $\Set{G^{ij}}$ be two sets of compatible fundamental matrices. They are equivalent under fundamental action if and only if they have solutions whose camera centers are equivalent under $\mathrm{PGL}_4$.
\end{proposition}

For the proof we need the following lemma:
\begin{lemma} [{\hspace{1sp}\cite[Result 22.1]{Hartley2004}}]
\label{lem:cameras_sharing_center_are_equivalent}
Let $P$ and $P'$ be two camera matrices with the same center. Then there
exists $H\in\PGL_3$ such that $P' = HP$.
\end{lemma}

\begin{proof}[Proof of \Cref{prop:equivalence}] $ $

$\Rightarrow)$ Let $G^{ij}=H_i^TF^{ij}H_j$. If $P_1,\ldots,P_n$ is a solution to $\Set{F^{ij}}$, then by \Cref{prop:fundamental_action_preserves_compibility}, $H_1^{-1}P_1,\ldots,H_n^{-1}P_n$ is a solution to $G^{ij}$, which have the same centers as $P_1,\ldots,P_n$. 

$\Leftarrow)$ Let $P_1,\ldots,P_n$ be a solution to $\Set{F^{ij}}$ with centers $c_i$ and $P_1',\ldots,P_n'$ a solution to $\Set{G^{ij}}$ with centers $c_i'$ such that $c_i'=H^{-1}c_i$ for some $H\in \mathrm{PGL}_4$. By \Cref{lem:cameras_sharing_center_are_equivalent}, there are $H_i\in \mathrm{PGL}_3$ such that $P_i'=H_iP_iH$, since $P_i'$ and $P_iH$ have the same center $H^{-1}c_i$. Then by \Cref{prop:fundamental_action_preserves_compibility}, $\Set{F^{ij}}$ and $\Set{G^{ij}}$ are equivalent under fundamental action.
\end{proof}

In this paper, quantities of the form $\textbf{e}_{sijt}:=(e_i^s)^TF^{ij}e_j^t$, called \emph{epipolar numbers}, are important (see \Cref{thm:compatible_forms,thm: 4tuple-condition}). The epipolar numbers are invariant under the fundamental action: 
\begin{lemma}\label{le: epip inv} Let $\Set{{F}^{ij}}$ be a set of fundamental matrices with epipoles $\Set{{e}_j^i}$. Let ${H}_i\in \mathrm{GL}_3^n$ and consider the fundamental matrices ${G}^{ij}:={H}_i^T{F}^{ij}{H}_j$, whose epipoles are ${h}_j^i={H}_j^{-1}{e}_j^i$. Then
\begin{align}
({e}_i^s)^T{F}^{ij}{e}_j^t=({h}_i^s)^T{G}^{ij}{h}_j^t.
\end{align}
\end{lemma}

\begin{proof} The equality follows directly by the definitions of $G^{ij}$ and $h_j^i$.
\end{proof}


We have the following geometrical interpretation of the epipolar numbers. 

\begin{lemma}\label{le: epip-coplane} Let $\Set{F^{ij}}$ be set of compatible fundamental matrices that include $F^{si},F^{ij}$ and $F^{jt}$. We have $\textnormal{\textbf{e}}_{sijt}=0$ if and only if the centers $c_s,c_i,c_j$ and $c_t$ of any solution are coplanar. 
\end{lemma}

The \emph{back-projected line} of an image point $x$ for a camera $P$ is the line in $\PP^3$ of all points that are projected by $P$ to $x$. This line contains the center of $P$. 
\begin{proof} Let $P_1,\ldots,P_n$ be a solution to $\Set{F^{ij}}$. Let $L_{i,s}$ be the back-projected line of $e_{i}^s$ and $L_{j,t}$ the back-projected line of $e_j^t$. Then $e_i^sF^{ij}e_j^t=0$ means precisely that the back-projected lines $L_{i,s}$ and $ L_{j,t}$ meet in a point. Therefore, $L_{i,s}$ and $L_{j,t}$ together span a plane unless they are the same line. In either case, all centers lie in this span, since $L_{i,s}$ contains $c_i$ and $c_s$, and $L_{j,t}$ contains $c_j$ and $c_t$. The other direction follows similarly. 
\end{proof}

It follows from the lemma that putting any of the two indices $s,i,j,t$ equal, the epipolar number is zero. In particular, $\textbf{e}_{sijs}$ is always zero for compatible fundamental matrices, because three centers are always in a plane. 

\section{Compatibility for Complete Graphs}\label{s: Kn} 
We begin by giving our main results for complete graphs, that is, the case where all the fundamental matrices are known. The main contribution of this paper is providing explicit, algebraic conditions for compatibility expressed in terms of the fundamental matrices and their epipoles for any number of views. Let $K_n$ denote the complete graph on $n$ nodes.

In \Cref{ss: K3}, we deal with $K_3$ graphs and recall the triple-wise conditions. We also state a result for the collinear case. In \Cref{ss: K4} we find necessary and sufficient constraints for compatibility in the case of $K_4$. In \Cref{ss: Kn} we prove that quadruple-wise compatbility implies global compatibility. Finally, in \Cref{ss: n-view} we state that the eigenvalue condition from the theorem of Kasten et. al. is redundant in the generic and collinear cases.

\begin{remark}
In this section, we work only with real numbers, because it allows us to give polynomials equations using the standard inner product and norm on $\RR^3$. However, all of our statements in \Cref{ss: K3} and \Cref{ss: K4} can be extended to the complex numbers.
\end{remark}


\subsection{$K_3$} \label{ss: K3}
The case of three fundamental matrices is fairly straightforward. We have two possible configurations for the three camera centers; they either all lie on a line, or they do not.
\begin{theorem}[{\hspace{1sp}\cite[Section 15.4]{Hartley2004}}]
\label{thm:compatible_forms}
Let $F^{12}$, $F^{13}$, $F^{23}$ be fundamental matrices. There exist non-collinear cameras $P_1,P_2,P_3$ such that $F^{ij}=\psi(P_i,P_j)$ if and only if
\begin{align}
\label{eq_non_collinear}
e_{1}^{2}\neq e_{1}^{3}, \quad e_{2}^{1}\neq e_{2}^{3}, \quad e_{3}^{1}\neq e_{3}^{2},
\end{align}
and
\begin{align}
\label{eq_compatible}
(e_{1}^{3})^{T}F^{12}e_{2}^{3}=(e_{1}^{2})^{T}F^{13}e_{3}^{2}=(e_{2}^{1})^{T}F^{23}e_{3}^{1}=0.
\end{align}
\end{theorem}

If $P_1,P_2,P_3$ are cameras with collinear centers, then it follows that $P_i(\ker P_j)=P_i(\ker P_k)$ for all distinct $i,j,k$. This implies that for the corresponding fundamental matrices $F^{12},F^{13},F^{23}$, we have $e_j^i=e_j^k$ for all distinct $i,j,k$. However, contrary to what is claimed in \cite{HK}, the conditions in \Cref{eq_compatible} are not enough in this case: 

\begin{example}
\label{ex:counterexample1}
\textnormal{Consider the fundamental matrices:}
\begin{align}
\begin{aligned}
    &F^{12}=\begin{bmatrix} 0&0&0\\0&1&0\\0&0&1 \end{bmatrix}, & &F^{13}=\begin{bmatrix} 0&0&0\\0&0&1\\0&1&0 \end{bmatrix},    & &F^{23}=\begin{bmatrix} 0&0&0\\0&1&1\\0&-1&1 \end{bmatrix},  
\end{aligned}
\end{align}
\textnormal{whose epipoles are all equal to $[1,0,0]$. These six matrices satisfy the conditions in \Cref{eq_compatible}. However, no solution of cameras $P_1,P_2,P_3$ exist for which $F^{12},F^{13},F^{23}$ are the fundamental matrices. This can be checked for instance via the algorithm described at the end of \Cref{s: Pre}.}
\hfill$\diamond$
\end{example}

Given a vector $t\in \RR^3$, we define 
\begin{align}
    [t]_\times=\begin{bmatrix} 0 & -t_3& t_2\\
t_3 & 0 & -t_1\\ -t_2 &t_1 &0\end{bmatrix}.
\end{align}
Then with respect to the cross product $\times$ on $\RR^3\times \RR^3$, we have $t\times u=[t]_\times u$. To the best of our knowledge, the following result does not appear in the literature:

\begin{proposition}\label{prop: K3 colin} Let $F^{12}$, $F^{13}$, $F^{23}$ be fundamental matrices. There exist collinear cameras $P_1,P_2,P_3$ such that $F^{ij}=\psi(P_i,P_j)$ if and only if
\begin{align}
\label{eq_collinear}
e_{1}^{2}= e_{1}^{3}, \quad e_{2}^{1}= e_{2}^{3}, \quad e_{3}^{1}= e_{3}^{2},
\end{align}  
and (up to scaling)
\begin{align}
    (F^{12})^T[e_1^2]_\times F^{13}=F^{23}.
\end{align}
\end{proposition}

The conditions of \Cref{thm:compatible_forms,prop: K3 colin} are called the \emph{triple-wise conditions.}

\begin{remark} When we in the proofs below write ``it can be verified that'' or ``it can be checked that'' in relation to the shape of fundamental matrices, we have checked this fact in \texttt{Macaulay2}.
\end{remark}

\begin{proof} Recall that the epipole $e_j^i$ equals $P_j(\ker(P_i))$. It follows that if a solution to $F^{12},F^{13},F^{23}$ consists of collinear cameras, then \Cref{eq_collinear} must be satisfied. Conversely, if \Cref{eq_collinear} is satisfied, any solution must consist of collinear camera centers.

We begin by simplifying the problem using the fundamental action. Let
\begin{align}
    H_i=\begin{bmatrix}
   e_i^k\,\, \textbf{x}_i\,\, \textbf{y}_i
    \end{bmatrix},
\end{align} for any $k\neq i$ and $\textbf{x}_i,\textbf{y}_i\in \RR^3$ such that the determinant is non-zero, meaning $H_i$ is invertible. We get a new triple of fundamental matrices
\begin{align}
G^{ij}=H_i^TF^{ij}H_j.
\end{align}
Write $h_j^i$ for the epipoles of $G^{ij}$. By the fact that $H_j^{-1}e_j^i$ spans $\ker G^{ij}$ we have $h_j^i=H_j^{-1}e_j^i$ (up to scaling). By construction of $H_j$, we then have:
\begin{align}
\begin{aligned}
    &h_1^2=[1,0,0], & &h_2^1=[1,0,0],  & &h_3^1=[1,0,0], \\
    &h_1^3=[1,0,0], & &h_2^3=[1,0,0], & &h_3^2=[1,0,0]. 
\end{aligned}
\end{align}
Since the epipoles span the kernels of $G^{ij}$, we conclude that $G^{ij}$ take the following form
\begin{align} \label{eq: Gform}
\begin{aligned}
    &G^{12}=\begin{bmatrix} 0&0&0\\0&a_{12}&b_{12}\\0&c_{12}&d_{12} \end{bmatrix}, & &G^{13}=\begin{bmatrix} 0&0&0\\0&a_{13}&b_{13}\\0&c_{13}&d_{13} \end{bmatrix}, &      &G^{23}=\begin{bmatrix} 0&0&0\\0&a_{23}&b_{23}\\0&c_{23}&d_{23}  \end{bmatrix},  
\end{aligned}
\end{align}
for some $a_{ij},b_{ij},c_{ij},d_{ij}\in \RR$ making them rank-2. 

We next find conditions on triplets of cameras $P_1,P_2,P_3$ with collinear centers whose fundamental matrices are of the form given by \Cref{eq: Gform}. We may up to $\mathrm{PGL}_4$ action assume that the center of $P_1$ is $[1,0,0,0]$, the center of $P_2$ is $[0,1,0,0]$ and the center of $P_3$ is $[1,1,0,0]$. Fix $P_1$ to be
\begin{align}
        P_1&=\left[\begin{matrix}
        0&1&0&0\\
        0&0&1&0\\
        0&0&0&1
    \end{matrix}\right].
\end{align}
Using the fact that $e_j^i=P_j(\ker P_i)$, we find that $P_2$ and $P_3$ must take the following form:
\begin{align}
    P_2=\left[\begin{matrix} 
        1&0&*&*\\
        0&0&*&*\\
        0&0&*& *    \end{matrix}\right],
     P_3&=\left[\begin{matrix}
        1&-1&*&*\\
        0&0&*&*\\
        0&0&*&*
    \end{matrix}\right].
    \end{align}
One can check that the two right-most elements of the first rows of $P_2$ and $P_3$ do not affect the fundamental matrices. In particular, if $G^{ij}$ are compatible, then one solution must be 
\begin{align}
    P_2=\left[\begin{matrix} 
        1&0&0&0\\
        0&0&\alpha_1&\alpha_2\\
        0&0&\alpha_3&\alpha_4    \end{matrix}\right],
     P_3=\left[\begin{matrix}
        1&-1&0&0\\
        0&0&\beta_1&\beta_2\\
        0&0&\beta_3&\beta_4
    \end{matrix}\right],
    \end{align}
for $\alpha_i$ and $\beta_i$ such that $\alpha_1\alpha_4-\alpha_2\alpha_3\neq 0$ and $\beta_1\beta_4-\beta_2\beta_3\neq 0$. Given such cameras, the fundamental matrices are calculated as  
\begin{align}\begin{aligned}\label{eq: psi P123}
    \psi(P_1,P_2)&=\left[\begin{matrix} 
        0&0&0\\
        0&-\alpha_3&\alpha_1\\
        0&-\alpha_4&\alpha_2    \end{matrix}\right], 
     \psi(P_1,P_3)=\left[\begin{matrix}
        0&0&0\\
        0&-\beta_3&\beta_1\\
        0&-\beta_4&\beta_2
    \end{matrix}\right], \\     \psi(P_2,P_3)&=\left[\begin{matrix} 
        0&0&0\\
        0&-\alpha_4\beta_3+\alpha_3\beta_4&\alpha_4\beta_1-\alpha_3\beta_2\\
        0&\alpha_2\beta_3-\alpha_1\beta_4&-\alpha_2\beta_1+\alpha_1\beta_2    \end{matrix}\right].    
\end{aligned}
    \end{align}
    
Define the $\star$ operator on $2\times 2$ matrices as
\begin{align}
\begin{aligned}
    &\begin{bmatrix}
        v_1 & v_2 \\ v_3 & v_4
    \end{bmatrix}\star\begin{bmatrix}
        w_1 & w_2 \\ w_3 & w_4
    \end{bmatrix}
    :=\begin{bmatrix}
        v_3w_1-v_1w_3 & v_3w_2-v_1w_4\\ v_4w_1-v_2w_3 & v_4w_2-v_2w_4
    \end{bmatrix}
    =\begin{bmatrix}
        v_3 & v_1 \\ v_4 & v_2
    \end{bmatrix}\begin{bmatrix}
        w_1 & w_2 \\ -w_3 & -w_4
    \end{bmatrix}.
\end{aligned}
\end{align}
Then, by \Cref{eq: psi P123}, $G^{ij}$ on the form \Cref{eq: psi P123} are compatible if and only if (up to scaling) we have
\begin{align}
\begin{aligned}\label{eq: compat aij}
    &\begin{bmatrix}
        a_{12} & b_{12} \\ c_{12} & d_{12}
    \end{bmatrix}\star\begin{bmatrix}
        a_{13} & b_{13} \\ c_{13} & d_{13}
    \end{bmatrix}
    = \begin{bmatrix}
        a_{23} & b_{23} \\ c_{23} & d_{23}
    \end{bmatrix}.
\end{aligned}
\end{align}
By the construction of our fundamental action, we have
\begin{align}
\begin{aligned}
    a_{ij}&=[0,1,0]G^{ij}[0,1,0]^T=\textbf{x}_i^TF^{ij}\textbf{x}_j, &
    b_{ij}&=[0,1,0]G^{ij}[0,0,1]^T=\textbf{x}_i^TF^{ij}\textbf{y}_j,\\
    c_{ij}&=[0,0,1]G^{ij}[0,1,0]^T=\textbf{y}_i^TF^{ij}\textbf{x}_j, &
    d_{ij}&=[0,0,1]G^{ij}[0,0,1]^T=\textbf{y}_i^TF^{ij}\textbf{y}_j.
\end{aligned}
\end{align}
In the below, and throughout this section, we skip the transpose notation and write for instance $\textbf{x}_iF^{ij}\textbf{x}_j$ instead of $\textbf{x}_i^TF^{ij}\textbf{x}_j$. We get
\begin{align}
\begin{aligned}
    \begin{bmatrix}
        \textbf{x}_1F^{12}\textbf{x}_2 & \textbf{x}_1F^{12}\textbf{y}_2 \\ \textbf{y}_1F^{12}\textbf{x}_2  & \textbf{y}_1F^{12}\textbf{y}_2
    \end{bmatrix}\star&\begin{bmatrix}
\textbf{x}_1F^{13}\textbf{x}_3 & \textbf{x}_1F^{13}\textbf{y}_3 \\ \textbf{y}_1F^{13}\textbf{x}_3 & \textbf{y}_1F^{13}\textbf{y}_3
    \end{bmatrix}    =\begin{bmatrix}
        \textbf{x}_2F^{23}\textbf{x}_3 &  \textbf{x}_2F^{23}\textbf{y}_3\\  \textbf{y}_2F^{23}\textbf{x}_3  & \textbf{y}_2F^{23}\textbf{y}_3
    \end{bmatrix}.
\end{aligned}
\end{align}
However,
\begin{align}
    \begin{bmatrix}
    \textbf{x}_iF^{ij}\textbf{x}_j & \textbf{x}_iF^{ij}\textbf{y}_j\\  \textbf{y}_iF^{ij}\textbf{x}_j  & \textbf{y}_iF^{ij}\textbf{y}_j
    \end{bmatrix}=\begin{bmatrix}
        \textbf{x}_i^T\\ \textbf{y}_i^T\end{bmatrix}F^{ij}\begin{bmatrix}
        \textbf{x}_j & \textbf{y}_j
    \end{bmatrix},
\end{align}
and therefore, 
   \begin{align}
   \begin{aligned}
    &\begin{bmatrix}
        \textbf{x}_1F^{12}\textbf{x}_2 & \textbf{x}_1F^{12}\textbf{y}_2  \\  \textbf{y}_1F^{12}\textbf{x}_2 & \textbf{y}_1F^{12}\textbf{y}_2
    \end{bmatrix}\star\begin{bmatrix}
\textbf{x}_1F^{13}\textbf{x}_3 &  \textbf{x}_1F^{13}\textbf{y}_3\\ \textbf{y}_1F^{13}\textbf{x}_3  & \textbf{y}_1F^{13}\textbf{y}_3
    \end{bmatrix}\\
    =&\begin{bmatrix}
        \textbf{x}_2^T\\ \textbf{y}_2^T\end{bmatrix}F^{21}\begin{bmatrix}
        \textbf{y}_1 & \textbf{x}_1
    \end{bmatrix}\begin{bmatrix}
        \textbf{x}_1^T\\ -\textbf{y}_1^T\end{bmatrix}F^{13}\begin{bmatrix}
        \textbf{x}_3 & \textbf{y}_3
    \end{bmatrix}\\
    =& \begin{bmatrix}
        \textbf{x}_2^T\\ \textbf{y}_2^T\end{bmatrix}F^{23}\begin{bmatrix}
        \textbf{x}_3 & \textbf{y}_3
    \end{bmatrix}.
   \end{aligned}
\end{align} 
Since this holds for generic choices of $\textbf{x}_2,\textbf{y}_2,\textbf{x}_3,\textbf{y}_3$, we conclude that, projectively,
   \begin{align}
    F^{21}\begin{bmatrix}
        \textbf{y}_1 & \textbf{x}_1
    \end{bmatrix}\begin{bmatrix}
        \textbf{x}_1^T\\ -\textbf{y}_1^T\end{bmatrix}F^{13}=F^{23},
\end{align} 
for all $\textbf{x}_1,\textbf{y}_1$ such that $[e_1^2\; \textbf{x}_1\; \textbf{y}_1]$ is invertible. Further, 
   \begin{align}
    \begin{bmatrix}
        \textbf{y}_1 & \textbf{x}_1
    \end{bmatrix}\begin{bmatrix}
        \textbf{x}_1^T\\ -\textbf{y}_1^T\end{bmatrix}
\end{align} 
is skew-symmetric and equals $[\ell]_\times$ for $\ell=\textbf{x}_1\times \textbf{y}_1\in \RR^3$. Then choosing $\textbf{x}_1,\textbf{y}_1$ such that $\ell=e_1^2$, we have over the real numbers that $[e_1^2\; \textbf{x}_1\; \textbf{y}_1]$ is full-rank. In other words, 
   \begin{align}\label{eq: e12}
    F^{21}[e_1^2]_\times F^{13}=F^{23}
\end{align} 
is a necessary and sufficient condition for compatibility.
\end{proof}

\begin{remark} In the complex setting, it does not always suffice to put $\ell=e_1^2$, because it could be the case that $(e_1^2)^Te_1^2=0$. Then $\ell$ should be any vector such that $\ell^Te_1^2\neq 0$.
\end{remark}



\subsection{$K_4$}\label{ss: K4}

We start this section with a counterexample to the previous belief that triple-wise compatibility is enough to ensure full compatibility.
\begin{example}
\label{ex:counterexample2}
\textnormal{Consider the fundamental matrices:}
\begin{align}
\begin{aligned}
    &F^{12}=\begin{bmatrix} 0&0&0\\0&0&1\\0&1&0 \end{bmatrix}, & &F^{13}=\begin{bmatrix} 0&0&1\\0&0&0\\0&1&0 \end{bmatrix}, & &F^{14}=\begin{bmatrix} 0&0&1\\0&1&0\\0&0&0 \end{bmatrix},\\ &F^{23}=\begin{bmatrix} 0&0&1\\0&0&0\\1&0&0 \end{bmatrix}, & &F^{24}=\begin{bmatrix} 0&0&1\\1&0&0\\0&0&0 \end{bmatrix}, & &F^{34}=\begin{bmatrix} 0&1&0\\2&0&0\\0&0&0 \end{bmatrix},  
\end{aligned}
\end{align}
\textnormal{with epipoles:}
\begin{align}\begin{aligned}
    &{e}_1^2=[1,0,0], & &{e}_2^1=[1,0,0], & &{e}_3^1=[1,0,0], & &{e}_4^1=[1,0,0],\\  
    &{e}_1^3=[0,1,0], & &{e}_2^3=[0,1,0], & &{e}_3^2=[0,1,0], & &{e}_4^2=[0,1,0],\\        
    &{e}_1^4=[0,0,1], & &{e}_2^4=[0,0,1], & &{e}_3^4=[0,0,1], & &{e}_4^3=[0,0,1].
\end{aligned}
\end{align}
 \textnormal{It can easily be verified that these six matrices satisfy the conditions in \Cref{thm:compatible_forms}. Nonetheless, no solution exists. Any attempt to find four cameras will end up matching at most five of the six fundamental matrices. We will soon see that this is because the sextuple does not satisfy the conditions in \Cref{thm: 4tuple-condition}.}
\hfill$\diamond$
\end{example}

Before we get to the main results, we list the possible configurations of camera centers in the case of four cameras (six fundamental matrices). These are illustrated in \cref{fig:4cases}. By \Cref{prop:equivalence}, these correspond to the equivalence classes of compatible fundamental matrices. Each of these will be recognizable from the epipoles $e_i^j$:

\begin{enumerate}[leftmargin =3.5em]
    \item [Case 1:] Cameras are in generic position, meaning no plane contains all four centers. Epipoles are in generic position, meaning in each image, the three epipoles do not lie on a line.
    \item [Case 2:] All camera centers lie in the same plane, but no three lie on a line. In each image, the three epipoles are distinct and lie on a line.
    \item [Case 3:] Precisely three camera centers lie on a line. In the three corresponding images, the epipoles corresponding to the other two cameras are equal, with the third one different from these two. In the final image, the three epipoles are distinct and lie on a line. 
    \item [Case 4:] All four camera centers lie on a line. In each image, the three epipoles coincide.
\end{enumerate}

\begin{figure}
\begin{center}
\includegraphics[width = \textwidth]{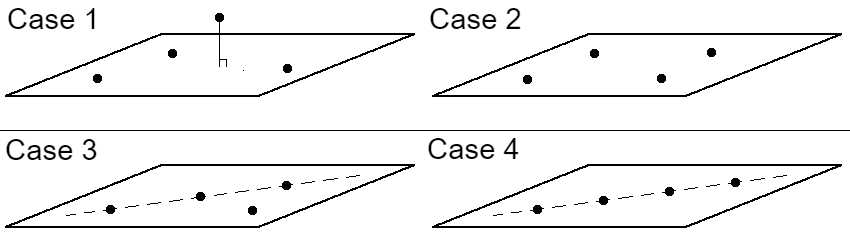}
\end{center}
\caption{Illustration of the 4 cases.}\label{fig:4cases}
\end{figure}

These are the only possible configurations of four cameras, so any compatible sextuple $\Set{F^{ij}}$ must have its epipoles in one of the configurations above. If we have, for instance, collinear epipoles in one image, but not all, the fundamental matrices can not be compatible. In Cases 3 and 4, the configuration of the epipoles together with the triple-wise conditions from \Cref{ss: K3} ensure compatibility. This is not true for Cases 1 and 2; here we need additional constraints. We cover all cases in sequence. We recall the epipolar numbers: $\textbf{e}_{sijt}=(e_i^s)^TF^{ij}e_j^t$.

\begin{theorem}[Case 1] 
\label{thm: 4tuple-condition} Let $\Set{F^{ij}}$ be a sextuple of fundamental matrices such that the three epipoles in each image do not lie on a line. 
Then $\Set{F^{ij}}$ is compatible if and only if the triple-wise conditions hold and 
\begin{align}
\begin{aligned} \label{eq: 4-tuple}
&\textnormal{\textbf{e}}_{4123}\textnormal{\textbf{e}}_{2134}\textnormal{\textbf{e}}_{3142}\textnormal{\textbf{e}}_{4231}\textnormal{\textbf{e}}_{1243}\textnormal{\textbf{e}}_{2341}=\textnormal{\textbf{e}}_{3124}\textnormal{\textbf{e}}_{4132}\textnormal{\textbf{e}}_{2143}\textnormal{\textbf{e}}_{1234}\textnormal{\textbf{e}}_{3241}\textnormal{\textbf{e}}_{1342}.
\end{aligned}
\end{align}
\end{theorem}
\begin{remark}  The condition that the epipoles in each image do not lie on a line is equivalent to all epipolar number $\textnormal{\textbf{e}}_{ijkl}$ being non-zero for distinct $i,j,k,l$. In Cases $2,3$ and $4$, the three epipoles in each image lie on a line. This is equivalent to all epipolar numbers $\textnormal{\textbf{e}}_{ijkl}$ being zero for distinct $i,j,k,l$. 
\end{remark}

\begin{proof} The triple-wise conditions are necessary for compatibility, so we assume that they are satisfied and prove that in this case compatibility is equivalent to \Cref{eq: 4-tuple} being satisfied. We begin by simplifying the problem. Let
\begin{align}
    H_i=\begin{bmatrix}
    e_i^j\,\, e_i^k\,\, e_i^l
    \end{bmatrix}.
\end{align}
This $3\times3$ matrix is of full-rank and takes the three coordinate points to the three epipoles in the $i$-th image. Using this as our fundamental action, we get a new sextuple of fundamental matrices
\begin{align}G^{ij}=H_i^TF^{ij}H_j.
\end{align}
Since the fundamental action preserves compatibility, the sextuple $\Set{G^{ij}}$ is compatible if and only if $\Set{F^{ij}}$ is. Note that the epipoles of $G^{ij}$, denoted by $h_j^i$, are:
\begin{align}
\begin{aligned}
    &{h}_1^2=[1,0,0], & &{h}_2^1=[1,0,0], & &{h}_3^1=[1,0,0], & &{h}_4^1=[1,0,0],\\  
    &{h}_1^3=[0,1,0], & &{h}_2^3=[0,1,0], & &{h}_3^2=[0,1,0], & &{h}_4^2=[0,1,0],\\        
    &{h}_1^4=[0,0,1], & &{h}_2^4=[0,0,1], & &{h}_3^4=[0,0,1], & &{h}_4^3=[0,0,1].
\end{aligned}
\end{align}
Moreover, since $G^{ij}$ satisfy the triple-wise conditions (we assumed $F^{ij}$ did, and these are preserved under fundamental action), it follows that the six matrices must be on the form: 
\begin{align}\label{eq: Gform c1}
\begin{aligned}
    &G^{12}=\begin{bmatrix} 0&0&0\\0&0&x_{12}\\0&y_{12}&0 \end{bmatrix},      &G^{13}=\begin{bmatrix} 0&0&x_{13}\\0&0&0\\0&y_{13}&0 \end{bmatrix},    &  &G^{14}=\begin{bmatrix} 0&0&x_{14}\\0&y_{14}&0\\0&0&0 \end{bmatrix},\\
    &G^{23}=\begin{bmatrix} 0&0&x_{23}\\0&0&0\\y_{23}&0&0 \end{bmatrix}, &G^{24}=\begin{bmatrix} 0&0&x_{24}\\y_{24}&0&0\\0&0&0 \end{bmatrix},    &  &G^{34}=\begin{bmatrix} 0&x_{34}&0\\y_{34}&0&0\\0&0&0 \end{bmatrix}. 
\end{aligned}
\end{align}
The sextuple $\Set{G^{ij}}$ is compatible if and only if there exists a reconstruction consisting of 4 cameras $P_i$. Since the epipoles do not lie on a line, any such reconstruction must have 4 linearly independent centers. We are free to choose coordinates in $\p3$ without affecting compatibility, so we take the four camera centers (assuming cameras exist) to be the four unit vectors. Furthermore, we know that the epipoles satisfy
\begin{align}
\label{eq:epipole}
    h_i^j=P_i(\ker(P_j)).
\end{align}
So if $\Set{G^{ij}}$ has a reconstruction $\Set{P_i}$, it must be on the form:
\begin{align}\begin{aligned}
    &P_1=\left[\begin{matrix} 0&\alpha_1^1&0&0\\0&0&\alpha_1^2&0\\0&0&0&\alpha_1^3 \end{matrix}\right], 
    & &P_2=\left[\begin{matrix} \alpha_2^1&0&0&0\\0&0&\alpha_2^2&0\\0&0&0&\alpha_2^3 \end{matrix}\right],\\
    &P_3=\left[\begin{matrix} \alpha_3^1&0&0&0\\0&\alpha_3^2&0&0\\0&0&0&\alpha_3^3 \end{matrix}\right],
    & &P_4=\left[\begin{matrix} \alpha_4^1&0&0&0\\0&\alpha_4^2&0&0\\0&0&\alpha_4^3&0 \end{matrix}\right],   
\end{aligned}
\end{align}
where $\alpha_i^j$ are scalars.
Since the fundamental matrices are of rank-2 and the cameras are rank-3, all the $\alpha_i^j$, as well as the $x_{ij}$ and $y_{ij}$ are non-zero. Computing the fundamental matrices of these four cameras, and setting them equal to the $G^{ij}$, we get the following six equations:
\begin{align}\label{eq: elimsystem}
\begin{aligned}
    &x_{12}\alpha_1^2\alpha_2^3=-y_{12}\alpha_1^3\alpha_2^2, & &x_{13}\alpha_1^1\alpha_3^3=-y_{13}\alpha_1^3\alpha_3^2, & &x_{14}\alpha_1^1\alpha_4^3=-y_{14}\alpha_1^2\alpha_4^2, \\
    &x_{23}\alpha_2^1\alpha_3^3=-y_{23}\alpha_2^3\alpha_3^1, & &x_{24}\alpha_2^1\alpha_4^3=-y_{24}\alpha_2^2\alpha_4^1, & &x_{34}\alpha_3^1\alpha_4^2=-y_{34}\alpha_3^2\alpha_4^1.
\end{aligned}
\end{align}
Eliminating the variables $\alpha_i^j$, we are left with a single polynomial,
\begin{align}
\label{eq:4_tuple_as_x_y}
    x_{12}y_{13}x_{14}x_{23}y_{24}x_{34}-y_{12}x_{13}y_{14}y_{23}x_{24}y_{34}=0.
\end{align}
This tells us that \Cref{eq: elimsystem} implies \Cref{eq:4_tuple_as_x_y}, and we are left to argue that if $x_{ij},y_{ij}$ are non-zero numbers such that \Cref{eq:4_tuple_as_x_y} holds, then there are non-zero $\alpha_i^j$ such that \Cref{eq: elimsystem} holds. Note that we can assume $\alpha_1^j=1$ by $\mathrm{PGL}_4$ action and that $\alpha_i^1=1$ by scaling. Writing $\lambda_{ij}=x_{ij}/y_{ij}$, we then aim to find non-zero $\alpha_i^j$ such that 
\begin{align}
\begin{aligned}
   \lambda_{12}\alpha_2^3&=\alpha_2^2, &  \lambda_{13}\alpha_3^3&=\alpha_3^2, & \lambda_{14}\alpha_4^3&=\alpha_4^2,\\
\lambda_{23}\alpha_3^3&=\alpha_2^3, &  \lambda_{24}\alpha_4^3&=\alpha_2^2, & \lambda_{34}\alpha_4^2&=\alpha_3^2.
\end{aligned}
\end{align}
It is clear that we can find non-zero $\alpha_i^j$ that solve the first five equations. However, this is enough because using $\lambda_{12}\lambda_{14}\lambda_{23}\lambda_{34}=\lambda_{13}\lambda_{24}$, the sixth equation $\lambda_{34}\alpha_4^2=\alpha_3^2$ is implied by the other five through substitution.

It follows that the set $\Set{G^{ij}}$ is compatible if and only if  \Cref{eq:4_tuple_as_x_y} is satisfied. Finally, we can express the $x_{ij}$ and $y_{ij}$ in terms of $F^{ij}$ and $e_i^j$, for instance we have
\begin{align}
\begin{aligned}
x_{12}&=(h_1^3)^TG^{12}h_2^4=(h_1^3)^TH_1^TF^{12}H_2h_2^4=(e_1^3)^TF^{12}e_2^4.
\end{aligned}
\end{align}
Making these substitutions for all the $x_{ij}$ and $y_{ij}$, we get \Cref{eq: 4-tuple}.
\end{proof}

\begin{theorem}[Case 2]\label{thm: Case 2} Let $\Set{F^{ij}}$ be a sextuple of fundamental matrices whose epipoles in each image are distinct and lie on a line. Then $\Set{F^{ij}}$ is compatible if and only if the triple-wise conditions hold,
\begin{align}
\begin{aligned}\label{eq: ijk F}
    \langle F^{jk}e_k^i,F^{jl}e_l^i \rangle \langle F^{kj}e_j^i,F^{kl}e_l^i \rangle\langle F^{lj}e_j^i,F^{lk}e_k^i \rangle +\|F^{lj}e_j^i\|^2\|F^{jk}e_k^i\|^2\|F^{kl}e_l^i\|^2 =0 
\end{aligned}
\end{align}
for all distinct $i,j,k,l$ satisfying $j<k<l$, and if for $\textbf{x}_i=F^{ij}e_j^l$ with $j<k<l$, we have
\begin{align}
\begin{aligned}
\label{eq:long_equation}
    -\frac{e_2^3F^{24}\textbf{x}_4}{\textbf{x}_2F^{24}e_4^1}\frac{\textbf{x}_1F^{12}\textbf{x}_2}{\textbf{x}_1F^{12}e_2^3}+  \frac{e_3^2F^{34}\textbf{x}_4}{\textbf{x}_3F^{34}e_4^1}\frac{\textbf{x}_1F^{13}\textbf{x}_3}{\textbf{x}_1F^{13}e_3^2} +\frac{e_3^1F^{34}\textbf{x}_4}{\textbf{x}_3F^{34}e_4^1}\frac{\textbf{x}_2F^{23}\textbf{x}_3}{\textbf{x}_2F^{23}e_3^1}+\\
    +\frac{e_3^2F^{34}\textbf{x}_4}{e_1^2F^{14}\textbf{x}_4}\frac{e_1^2F^{13}\textbf{x}_3}{\textbf{x}_1F^{13}e_3^2}\frac{\textbf{x}_1F^{14}\textbf{x}_4}{\textbf{x}_3F^{34}e_4^1}+\frac{\textbf{x}_2F^{24}\textbf{x}_4}{\textbf{x}_2F^{24}e_4^1}-\frac{\textbf{x}_3F^{34}\textbf{x}_4}{\textbf{x}_3F^{34}e_4^1}=0.
\end{aligned}
\end{align}
\end{theorem}

\begin{remark}
As \Cref{eq:long_equation} is already oversaturated with sub/superscript, we are omitting the transpose symbol from these equations. It is to be understood that the 3-vectors $\textbf{x}_i$ and $e_i^j$ are column-vectors when directly right of a fundamental matrix, and row-vectors when to the left.
\end{remark}

\begin{proof}
Like in the previous proof, we begin by assuming the triple-wise conditions are satisfied. The three epipoles in each image lie on a line and therefore we fix a scaling such that for each $i$ we have $e_i^l=e_i^j+e_i^k$, where $l>k>j$. Let
\begin{align}
H_i=\begin{bmatrix}
    e_i^j\,\, e_i^k \,\, \textbf{x}_i
    \end{bmatrix}.
\end{align}
Note that $(e_i^j)^T\textbf{x}_i$ and $(e_i^k)^T\textbf{x}_i$ for $\textbf{x}_i$ in the statement are both zero, so $H_i$ is of full-rank. Using this as our fundamental action, we get a new sextuple of fundamental matrices
\begin{align}    G^{ij}=H_i^TF^{ij}H_j.
\end{align}
Since the fundamental action preserves compatibility, the sextuple $\Set{G^{ij}}$ is compatible if and only if $\Set{F^{ij}}$ is. Note that the epipoles of $G^{ij}$ are as follows:
\begin{align}
\begin{aligned}
    &{h}_1^2=[1,0,0], & &{h}_2^1=[1,0,0], & &{h}_3^1=[1,0,0], & &{h}_4^1=[1,0,0],\\  
    &{h}_1^3=[0,1,0], & &{h}_2^3=[0,1,0], & &{h}_3^2=[0,1,0], & &{h}_4^2=[0,1,0],\\  
    &{h}_1^4=[1,1,0], & &{h}_2^4=[1,1,0], & &{h}_3^4=[1,1,0], & &{h}_4^3=[1,1,0].
\end{aligned}
\end{align}
With these epipoles and the fact that the $G^{ij}$ satisfy the triple-wise conditions (we assumed $F^{ij}$ did, and these are preserved under fundamental action), it follows that the six matrices must be on the form:
\begin{align}\label{eq: Gform c2}
\begin{aligned}
    &G^{12}=\left[\begin{matrix} 0&0&0\\0&0&x_{12}\\0&y_{12}&z_{12} \end{matrix}\right], & &G^{13}=\left[\begin{matrix} 0&0&x_{13}\\0&0&0\\0&y_{13}&z_{13} \end{matrix}\right], & &G^{14}=\left[\begin{matrix} 0&0&x_{14}\\0&0&-x_{14}\\0&y_{14}&z_{14} \end{matrix}\right],\\
    &G^{23}=\left[\begin{matrix} 0&0&x_{23}\\0&0&0\\y_{23}&0&z_{23} \end{matrix}\right], & &G^{24}=\left[\begin{matrix} 0&0&x_{24}\\0&0&-x_{24}\\y_{24}&0&z_{24} \end{matrix}\right],
    & &G^{34}=\left[\begin{matrix} 0&0&x_{34}\\0&0&-x_{34}\\y_{34}&-y_{34}&z_{34} \end{matrix}\right].
\end{aligned}
\end{align}
The sextuple $\Set{G^{ij}}$ is compatible if and only if there exists a reconstruction consisting of 4 cameras $P_i$ with centers that lie in a plane, but no three collinear, since the three epipoles are collinear in each image. 
We are free to choose coordinates in $\p3$ without changing the fundamental matrices, so we take the four camera centers (assuming they exist) to be $[1,0,0,0]$, $[0,1,0,0]$, $[0,0,1,0]$, and $[1,1,1,0]$. Furthermore, by the definition of the epipole, we know that the epipoles satisfy
\begin{align}
    h_i^j=P_i(\ker(P_j)). 
\end{align}
So if $\Set{G^{ij}}$ has a reconstruction $\Set{P_i}$, it must be on the form:
\begin{align}
\begin{aligned}
    &P_1=\left[\begin{matrix} 0&1&0&\alpha_1^1\\0&0&1&\alpha_1^2\\0&0&0&\alpha_1^3 \end{matrix}\right],  & &P_2=\left[\begin{matrix} 1&0&0&\alpha_2^1\\0&0&1&\alpha_2^2\\0&0&0&\alpha_2^3 \end{matrix}\right],\\
    &P_3=\left[\begin{matrix} 1&0&0&\alpha_3^1\\0&1&0&\alpha_3^2\\0&0&0&\alpha_3^3 \end{matrix}\right],  & &P_4=\left[\begin{matrix} 1&0&-1&\alpha_4^1\\0&1&-1&\alpha_4^2\\0&0&0&\alpha_4^3 \end{matrix}\right],
\end{aligned}
\end{align}
where the $\alpha_i^j$ are scalars. Since the fundamental matrices are of rank 2 and the cameras of rank 3, the four scalars $\alpha_i^3$, as well as all the $x_{ij}$ and $y_{ij}$ are non-zero. Computing the fundamental matrices of these four cameras, and setting them equal to the $G^{ij}$, we get the following set of equations: 
\begin{align}
\begin{aligned}
    &\frac{x_{12}}{y_{12}}=-\frac{\alpha_1^3}{\alpha_2^3},  & &\frac{x_{13}}{y_{13}}=-\frac{\alpha_1^3}{\alpha_3^3}, & &\frac{x_{14}}{y_{14}}=-\frac{\alpha_1^3}{\alpha_4^3},\\ &\frac{x_{23}}{y_{23}}=-\frac{\alpha_2^3}{\alpha_3^3},& &\frac{x_{24}}{y_{24}}=-\frac{\alpha_2^3}{\alpha_4^3}, & &\frac{x_{34}}{y_{34}}=-\frac{\alpha_3^3}{\alpha_4^3},
\end{aligned}
 \end{align}
and
\begin{align}
\begin{aligned}
    &\frac{z_{12}}{y_{12}}=\frac{\alpha_1^2-\alpha_2^2}{\alpha_2^3}, & &\frac{z_{13}}{y_{13}}=\frac{\alpha_1^1-\alpha_3^2}{\alpha_3^3},  & &\frac{z_{14}}{y_{14}}=\frac{\alpha_1^1-\alpha_1^2-\alpha_4^2}{\alpha_4^3},\\
    &\frac{z_{23}}{y_{23}}=\frac{\alpha_2^1-\alpha_3^1}{\alpha_3^3}, & &\frac{z_{24}}{y_{24}}=\frac{\alpha_2^1-\alpha_2^2-\alpha_4^1}{\alpha_4^3}, & &\frac{z_{34}}{y_{34}}=\frac{\alpha_3^1+\alpha_4^2-\alpha_3^2-\alpha_4^1}{\alpha_4^3}.
\end{aligned}
\end{align}
Eliminating the $\alpha_i^j$ from these equations gives us the following constraints:
\begin{align}
\label{eq:condition1}
    x_{jk}x_{kl}y_{jl}+y_{jk}y_{kl}x_{jl}=0 \quad \forall j<k<l,
\end{align}
and
\begin{align}
\label{eq:condition2}
    \frac{x_{24}}{y_{24}}\frac{z_{12}}{y_{12}}-\frac{x_{34}}{y_{34}}\frac{z_{13}}{y_{13}}+\frac{x_{34}}{y_{34}}\frac{z_{23}}{y_{23}}-\frac{z_{14}}{y_{14}}+\frac{z_{24}}{y_{24}}-\frac{z_{34}}{y_{34}}=0.
\end{align}
As in the proof of \Cref{thm: 4tuple-condition}, the fundamental matrices are compatible if and only if \Cref{eq:condition1,eq:condition2} are satisfied. Let $k$ be the smallest index satisfying $k\neq i,j$, then we can write
\begin{align}
    \begin{aligned}
        &x_{ij}=e_i^kF^{ij}\textbf{x}_j,\\
        &y_{ij}=\textbf{x}_iF^{ij}e_j^k,\\
        &z_{ij}=\textbf{x}_iF^{ij}\textbf{x}_j.\\
    \end{aligned}
\end{align}
With the substitution $\textbf{x}_i=F^{ij}e_j^l$ in \Cref{eq:condition1}, we get:
\begin{align}
\begin{aligned}
    &y_{jl}x_{jk}x_{kl}+x_{jl}y_{jk}y_{kl}\\
    =&(\textbf{x}_jF^{jl}e_l^i)(e_j^iF^{jk}\textbf{x}_k)(e_k^iF^{kl}\textbf{x}_l)+(e_j^iF^{jl}\textbf{x}_l)(\textbf{x}_jF^{jk}e_k^i)(\textbf{x}_kF^{kl}e_l^i)\\
    =&(e_k^iF^{kj}F^{jl}e_l^i)(e_j^iF^{jk}F^{kl}e_l^i)(e_j^iF^{jl}F^{lk}e_k^i)+(e_j^iF^{jl}F^{lj}e_j^i)(e_k^iF^{kj}F^{jk}e_k^i)(e_l^iF^{lk}F^{kl}e_l^i),\\
    =&\langle F^{jk}e_k^i,F^{jl}e_l^i \rangle \langle F^{kj}e_j^i,F^{kl}e_l^i \rangle\langle F^{lj}e_j^i,F^{lk}e_k^i \rangle+\|F^{lj}e_j^i\|^2\|F^{jk}e_k^i\|^2\|F^{kl}e_l^i\|^2 =0,
    \end{aligned}
\end{align}
hence we arrive at \Cref{eq: ijk F}. In \Cref{eq:condition2}, we use \Cref{eq:condition1} to substitute 
\begin{align}
  -\frac{1}{y_{14}}=\frac{x_{13}x_{34}}{x_{14}y_{13}y_{34}}  
\end{align}
and then plug in $\textbf{x}_i=F^{ij}e_j^l$ (we do this step to get a homogeneous equation in every fundamental matrix and epipole). This gives us \Cref{eq:long_equation}.
\end{proof}

\begin{remark} In the complex setting, we cannot always put $\textbf{x}_i=F^{ij}e_j^l$ in \Cref{thm: Case 2}, because there is no longer any guarantee that this makes $H_i$ invertible. For fixed complex $F^{ij}$, one can check if they are compatible in Case 2 instead by choosing any $\textbf{x}_i$ that make $H_i$ invertible. The same principle applies in Case 3. 
\end{remark}

\begin{theorem}[Case 3]\label{thm: Case 3} Let $\Set{F^{ij}}$ be a sextuple of fundamental matrices such that 
\begin{align}
    \begin{aligned}
        e_1^2=e_1^3\neq e_1^4,\,  e_2^1=e_2^3\neq e_2^4,\, e_3^1=e_3^2\neq e_3^4, 
    \end{aligned}
\end{align}
and $e_4^1,e_4^2,e_4^3$ are distinct and lie on a line. Then $\Set{F^{ij}}$ is compatible if and only if each triple is compatible.
\end{theorem}

\begin{proof} Like in the two previous proofs, we begin by assuming the triple-wise conditions are satisfied, since we know them to be necessary. Fix a scaling such that $e_4^3=e_4^1+e_4^2$. Let 
\begin{align}
    H_i=\begin{bmatrix}
    e_i^j\,\, e_i^l\,\, \textbf{x}_i
    \end{bmatrix}, \quad H_4=\begin{bmatrix}
    e_4^1\,\, e_4^2\,\, \textbf{x}_4
    \end{bmatrix}
\end{align}
for $i=1,2,3$ and $j<k<l$, and
\begin{align}  G^{ij}=H_i^TF^{ij}H_j.
\end{align}
Let $\textbf{x}_i=F^{ij}e_j^l$ with $j<k<l$ for $i=1,2,3$. Since all epipolar numbers are zero in this case, $(e_i^j)^T\textbf{x}_i$ and $(e_i^l)^T\textbf{x}_i$ are both zero. It follows that $H_i$ is full-rank for $i=1,2,3$. Let $\textbf{x}_4$ be such that $H_4$ is full-rank. The fundamental matrices $G^{ij}$ are compatible if and only if $F^{ij}$ are. Note that the epipoles of $G^{ij}$ are:
\begin{align*}
    &h_1^2=[1,0,0], & &\underline{h}_2^1=[1,0,0], & &h_3^1=[1,0,0], & &\underline{h}_4^1=[1,0,0],\\
    &h_1^3=[1,0,0], & &\underline{h}_2^3=[1,0,0], & &h_3^2=[1,0,0], & &h_4^2=[0,1,0],\\
    &\underline{h}_1^4=[0,1,0], & &\underline{h}_2^4=[0,1,0], & &\underline{h}_3^4=[0,1,0], & &\underline{h}_4^3=[1,1,0].
\end{align*}
With these epipoles and the fact that the $G^{ij}$ satisfy the triple-wise conditions (preserved under fundamental action), it follows that the six matrices must be on the form:
\begin{align}\label{eq: Gform c3}
\begin{aligned}
    &G^{12}=\left[\begin{matrix} 0&0&0\\0&0&x_{12}\\0&y_{12}&z_{12} \end{matrix}\right], & &G^{13}=\left[\begin{matrix} 0&0&0\\0&0&x_{13}\\0&y_{13}&z_{13} \end{matrix}\right],    & &G^{14}=\left[\begin{matrix} 0&0&x_{14}\\0&0&0\\0&y_{14}&z_{14} \end{matrix}\right],\\
    &G^{23}=\left[\begin{matrix} 0&0&0\\0&0&x_{23}\\0&y_{23}&z_{23} \end{matrix}\right], & &G^{24}=\left[\begin{matrix} 0&0&x_{24}\\0&0&0\\y_{24}&0&z_{24} \end{matrix}\right],    & &G^{34}=\left[\begin{matrix} 0&0&x_{34}\\0&0&0\\y_{34}&-y_{34}&z_{34} \end{matrix}\right]. 
\end{aligned}
\end{align}
The sextuple $\Set{G^{ij}}$ is compatible if and only if there exists a reconstruction consisting of 4 cameras $P_i$ with the centers of $P_1,P_2,P_3$ lying on a line that does not contain the center of $P_4$. To see this, note that the three epipoles in each image are collinear, implying that any reconstruction must consist of cameras with coplanar centers. Furthermore, since two epipoles coincide in the first three images, the centers of $P_1,P_2,P_3$ must lie on a line. We are free to choose coordinates in $\p3$ without changing the fundamental matrices, so we take the four camera centers (assuming they exist) to be $[1,0,0,0]$, $[0,1,0,0]$, $[1,1,0,0]$, and $[0,0,1,0]$. We recall that the epipoles satisfy
\begin{align}
    h_i^j=P_i(\ker(P_j)). 
\end{align}
So if $\Set{G^{ij}}$ has a reconstruction $\Set{P_i}$, it must be on the form:
\begin{align}\begin{aligned}
    &P_1=\left[\begin{matrix} 0&1&0&\alpha_1^1\\0&0&\beta_1&\alpha_1^2\\0&0&0&\alpha_1^3 \end{matrix}\right],       
    & &P_2=\left[\begin{matrix} 1&0&0&\alpha_2^1\\0&0&\beta_2&\alpha_2^2\\0&0&0&\alpha_2^3 \end{matrix}\right],\\
    &P_3=\left[\begin{matrix} 1&-1&0&\alpha_3^1\\0&0&\beta_3&\alpha_3^2\\0&0&0&\alpha_3^3 \end{matrix}\right],       
    & &P_4=\left[\begin{matrix} 1&0&0&\alpha_4^1\\0&1&0&\alpha_4^2\\0&0&0&\alpha_4^3 \end{matrix}\right].  
\end{aligned}
\end{align}
where the $\beta_i,\alpha_i^j$ are scalars. Since the fundamental matrices are rank-2 and the cameras rank-3, the four scalars $\alpha_i^3$, as well as all the $\beta_i$, $x_{ij}$ and $y_{ij}$ are non-zero. Computing the fundamental matrices of these four cameras, and setting them equal to the $G^{ij}$, we get after elimination the following two equations: 
\begin{align}\begin{aligned}\label{eq:condition3case3}
x_{12}x_{23}y_{13}+x_{13}y_{12}y_{23}&=0,\\
\frac{x_{23}}{y_{23}}\frac{z_{12}}{y_{12}}+\frac{z_{13}}{y_{13}}-\frac{z_{23}}{y_{23}}&=0.
\end{aligned}
\end{align}
Similarly to the proofs of Cases 1 and 2, \Cref{eq:condition3case3} are equivalent to $G^{ij}$ being compatible. We next observe that these equations precisely describe that $G^{12},G^{13}$ and $G^{23}$ are compatible. Indeed, we have seen in the proof of \Cref{prop: K3 colin} that for compatibility we must have (up to scale) 
\begin{align}
\label{eq:collinear_case_3}
    G^{23}=\begin{bmatrix} 0&0&0\\0&0&-y_{12}x_{13}\\0&x_{12}y_{13}&x_{12}z_{13}-x_{13}z_{12} \end{bmatrix}.
\end{align}
This is equivalent to
\begin{align}
    \mathrm{rank}\begin{bmatrix}
       -y_{12}x_{13} & x_{23}\\
       x_{12}z_{13}-x_{13}z_{12} & z_{23}\\
       x_{12}y_{13} & y_{23}
    \end{bmatrix}=1
\end{align} 
Setting the $2\times 2$ minors of this $3\times 2$ matrix to zero, we get a polynomial system equivalent to \Cref{eq:condition3case3}, finishing the proof. 
\end{proof}

\begin{theorem}[Case 4]\label{thm: Case 4} Let $\Set{F^{ij}}$ be a sextuple of fundamental matrices such that 
in each image, all three epipoles coincide. Then $\Set{F^{ij}}$ is compatible if and only if each triple is compatible.
\end{theorem}

This result is a direct consequence of \Cref{thm:compatible_if_each_sextuple_is_compatible}, proven in the next subsection.


\subsection{$K_n$}\label{ss: Kn}

For the case of more than 4 cameras, it turns out that quadruple-wise compatibility is sufficient to ensure global compatibility.

\begin{theorem}
\label{thm:compatible_if_each_sextuple_is_compatible}
Let $\Set{F^{ij}}$ be a complete set of $n\choose2$, $n\ge 4$, fundamental matrices such that 
for all $i,j,k,l$, the sextuple $F^{ij},F^{ik},F^{jk},F^{il},F^{jl},F^{kl}$ is compatible. Then $\Set{F^{ij}}$ is compatible.

Moreover, if all epipoles in each image coincide, then triple-wise compatibility implies that $\Set{F^{ij}}$ is compatible. The reconstruction in this case will be a set of cameras whose centers all lie on a line.
\end{theorem}

In the non-collinear case, we actually don't need to assume that all sextuples are compatible. It suffices that there is a sextuple $F^{12},F^{13},F^{14},F^{23},F^{24},F^{34}$ that is compatible with a solution of cameras $P_1,P_2,P_3,P_4$ such that the line spanned by the centers of $P_1,P_2$ do not contain the centers of $P_3,P_4$, and that each sextuple of fundamental matrices corresponding to indices $\{1,2,3,i\}$ and $\{1,2,4,i\}$ for $i\ge 5$ are compatible. This is what we show in the proof below. 



\begin{proof}
   We start with the collinear case. As in the proof of \Cref{prop: K3 colin}, it suffices to prove the statement for fundamental matrices
    \begin{align}
        G^{ij}=\begin{bmatrix}
        0&0&0\\
        0&a_{ij}&b_{ij}\\
        0&c_{ij}&d_{ij}
    \end{bmatrix}.
    \end{align}
    By the compatibility of $\Set{G^{1i},G^{1j},G^{ij}}$, we have by \Cref{prop: K3 colin} that
    \begin{align}
        G^{ij}=\begin{bmatrix}
        0&0&0\\
        0&c_{1i}a_{1j}-a_{1i}c_{1j}&c_{1i}b_{1j}-a_{1i}d_{1j}\\
        0&d_{1i}a_{1j}-b_{1i}c_{1j}&d_{1i}b_{1j}-b_{1i}d_{1j}
    \end{bmatrix},
    \end{align}
    for all $i,j\neq1$. It can be verified that the following cameras $P_i$ form a reconstruction of these fundamental matrices:
    \begin{align}\begin{aligned}
        P_1&=\begin{bmatrix}
        0&1&0&0\\
        0&0&1&0\\
        0&0&0&1
    \end{bmatrix}, &
    & P_i=\begin{bmatrix} 
        \gamma_i&1&0&0\\
        0&0&b_{1i}&d_{1i}\\
        0&0&-a_{1i}&-c_{1i}
    \end{bmatrix}, \forall i\neq1, 
    \end{aligned}
    \end{align}
    where $\gamma_i\neq 0$ are distinct numbers. Hence the $n\choose2$-tuple is compatible whenever each triple is compatible. We also observe that all cameras have a center lying on the line $[\lambda_1,\lambda_2,0,0]$.

   Now assume that in some image, not all epipoles coincide. We prove the theorem for the case $n=5$ and note that the principle extends to any $n$. 
   
   Consider a sextuple $S_{1234}=\Set{F^{12},F^{13},F^{14},F^{23},F^{24},F^{34}}$, where in some image, not all epipoles coincides. Let $P_1,P_2,P_3,P_4$ be a solution and without loss of generality assume that the line spanned by the centers of $P_1,P_2$ do not contain the centers of $P_3,P_4$. Let $P_1',P_2',P_3',P_5$ be a solution to $S_{1235}=\Set{F^{12},F^{13},F^{15},F^{23},F^{25},F^{35}}$. By \Cref{lem:triple_has_unique_solution}, we have that $P_1,P_2,P_3$ and $P_1',P_2',P_3'$ differ by $\mathrm{PGL}_4$, and we may therefore take them to be equal. 
   
   It remains to prove that $F^{45}$ is the fundamental matrix of $P_4,P_5$. For this we note that either 1) $P_1,P_2,P_5$ or 2) $P_1,P_3,P_5$ are not collinear cameras, since $P_1,P_2,P_3$ are not collinear. In the first case 1), consider the tuple $S_{1245}=\Set{F^{12},F^{14},F^{15},F^{24},F^{25},F^{45}}$ with solution $P_1'',P_2'',P_4'',P_5''$. By \Cref{lem:triple_has_unique_solution}, the overlap between $S_{1235}$ and $S_{1245}$ imply that we can via $\mathrm{PGL}_4$ action assume $P_1''=P_1,P_2''=P_2,P_5''=P_5$, and the overlap between $S_{1234}$ and $S_{1245}$ imply that we can also assume $P_4''=P_4$, since $P_1,P_2,P_4$ are not collinear. But since $F^{45}$ is the fundamental matrix of $P_4'',P_5''$ we conclude that it is also the fundamental matrix of $P_4,P_5$. In the second case 2) the argument is analogous when we consider $S_{1345}$ instead of $S_{1245}$.  
\end{proof}


While uniqueness is not the focus of this paper, we give the following useful theorem on the complete graph:

\begin{proposition}
\label{prop:unique_solution}
    A compatible set of $n\choose2$ fundamental matrices has a unique solution up to action by $\PGL(4)$, unless all the epipoles in each image are equal.
\end{proposition}

\begin{proof} If the set of fundamental matrices is compatible, and the epipoles in each image are not all equal, we know that there exists a reconstruction consisting of $n$ cameras, not all lying on a line. It follows from the Sylvester-Gallai theorem \cite[Chapter 11]{aigner1999proofs} that there will always be at least two cameras $P_1,P_2$ such that the line spanned by their camera centers does not contain any other camera centers. By \Cref{lem:triple_has_unique_solution}, a triple of compatible fundamental matrices has a unique solution if the two epipoles in each image are distinct, or equivalently if their reconstruction consists of three non-collinear cameras. Up to projective transformation, we can uniquely recover $P_1,P_2$ from $F^{12}$, which fixes coordinates in $\p3$. All other cameras $P_i$ are then uniquely determined by the triple $F^{12},F^{1i},F^{2i}$. Since this uniquely determines all cameras (up to global projective transformation), the fundamental matrices $F^{ij}$ can only have one solution.

Conversely, in the case that all epipoles in each image are equal, the constructed solution of cameras $P_i$ in the proof of \Cref{thm:compatible_if_each_sextuple_is_compatible} shows that there is no unique reconstruction of the centers (up to global projective transformation). This is because the choice of distinct numbers $\gamma_i$, was arbitrary. 
\end{proof}

\subsection{$n$-view matrices}\label{ss: n-view}

The compatibility of $n\choose 2$ fundamental matrices $F^{ij}$ was also studied in \cite{kasten2019gpsfm,geifman2020averaging} and we recall their results below. Given a set of $n\choose2$ fundamental matrices $F^{ij}$, the \emph{$n$-view fundamental matrix} is the $3n\times 3n$ symmetric matrix
\begin{align}
    \textbf{F}:=\begin{bmatrix} 0 & F^{12} & \cdots & F^{1n}\\
    F^{21} & 0 & \cdots & F^{2n}\\
    \vdots & \vdots & \ddots & \vdots\\
    F^{n1} & F^{n2} & \cdots & 0\\
    \end{bmatrix}.
\end{align}

\begin{theorem}[Theorem 1 of \cite{kasten2019gpsfm}, Theorem 2 of \cite{geifman2020averaging}]\label{thm: KGGB} Let $\Set{F^{ij}}$ be a complete set of $n\choose 2$ real fundamental matrices, where $n\ge 3$. Then $\Set{F^{ij}}$ is compatible with a solution of real cameras whose centers are not all collinear if and only if there exist non-zero scalars $\lambda_{ij}=\lambda_{ji}$ such that:
\begin{enumerate}
    \item the $n$-view fundamental matrix $~\textnormal{\textbf{F}}=(\lambda_{ij}F^{ij})_{ij}$ is rank-6 and has exactly three positive and three negative eigenvalues;
    \item the $3\times 3n$ and $3n\times 3$ block rows and block columns of \textnormal{$\textbf{F}$} are all of rank 3.
\end{enumerate}

Further, $\Set{F^{ij}}$ is compatible with a solution of real cameras whose centers are all collinear if and only if there exist non-zero scalars $\lambda_{ij}=\lambda_{ji}$ such that:
\begin{enumerate}
    \item the $n$-view fundamental matrix $~\textnormal{\textbf{F}}=(\lambda_{ij}F^{ij})_{ij}$ is rank-4 and has exactly two positive and two negative eigenvalues;
    \item the $3\times 3n$ and $3n\times 3$ block rows and block columns of \textnormal{$\textbf{F}$} are all of rank 2.
\end{enumerate}
\end{theorem}

Our work regarding the $K_3$ and $K_4$ cases can be used to improve on this result by showing that the eigenvalue condition can be dropped in the cases below.

\begin{theorem}
\label{thm: redundant} In the collinear case of \Cref{thm: KGGB}, the eigenvalue condition can be dropped. In the non-collinear case, the eigenvalue condition can be dropped if in each image, no three epipoles lie on a line.
\end{theorem}

\begin{proof} The structure of the proof is as follows. We prove in detail the when $n=3$ and sketch $n=4$ for Case 1. The \texttt{Macaulay2} code used in all these settings is attached. Then, we use \Cref{thm:compatible_if_each_sextuple_is_compatible} to argue that the general setting is implied by these case studies. 

We start with $n=3$ in the collinear setting. Let $F^{ij}$ be three fundamental matrices for which there exists a scaling $\lambda$ such that 
\begin{align}
    \begin{bmatrix}
        0 & F^{12} & F^{13} \\
        F^{21} & 0 & \lambda F^{23}\\
        F^{31} & \lambda F^{32} & 0
    \end{bmatrix}
\end{align}
is rank-4 and the $3\times 6$ and $6\times 3$ block rows and colums are rank-2. Note that we don't need to scale $F^{12}$ and $F^{21}$ or $F^{13}$ and $F^{31}$, because scaling each row and each column does not change the rank of the $3$-view matrix, so we may choose their scalings to be 1 without loss of generality. By the latter condition, $F^{12}$ and $F^{13}$ must have the same epipoles. We can say even more, namely that 
\begin{align}
e_{1}^{2}= e_{1}^{3}, \quad e_{2}^{1}= e_{2}^{3}, \quad e_{3}^{1}= e_{3}^{2}.
\end{align}  
As in the proof of \Cref{prop: K3 colin}, this assumption allows us to assume via fundamental action $F^{ij}$ take the form $G^{ij}$ of \Cref{eq: Gform}. We work in the polynomial ring $R=\QQ[a_{ij},b_{ij},c_{ij},d_{ij},\lambda]$, where $1\le i<j\le 3$ consider the following $3$-view matrix:
\begin{align}
    \textbf G(\lambda):=\begin{bmatrix}
        0 & G^{12} & G^{13} \\
        G^{21} & 0 & \lambda G^{23}\\
        G^{31} & \lambda G^{32} & 0
    \end{bmatrix}.
\end{align}
The rank of $\textbf G(\lambda)$ is at most 4 if and only if all $5\times 5$ minors of $\textbf G(\lambda)$ vanish and we therefore consider the ideal $I_{\mathrm{minors}}$ in $R$ defined by the $5\times 5$ minors of $\textbf G(\lambda)$. Since we don't want solutions with $\lambda=0$ or $\mathrm{rank} G^{ij}<2$, we saturate $I_{\mathrm{minors}}$ with respect to the ideals $I_\lambda=\langle \lambda\rangle$ and $ I_{ij}=\langle a_{ij}d_{ij}-b_{ij}c_{ij}\rangle$. After this is done in \texttt{Macaulay2}, we get a new ideal $I_{\mathrm{rank}}$ in $R$ with nine generators. 

Write $G^{ij'}$ for the matrices we get by removing the first row and column from $G^{ij}$. Recall that $G^{ij}$ as in \Cref{eq: Gform} are compatible if and only if they are rank-2 and up to scaling, $G^{12'}\star G^{13'}=G^{23'}$, i.e. \Cref{eq: compat aij} holds. As in the proof of \Cref{thm: Case 2}, this equality is described by vectorizing $G^{12'}\star G^{13'}$ and $G^{23'}$, putting them into a $4\times 2$ matrix and setting the rank to 1. By doing this we get 6 $2\times 2$ minors and we let $J_{\mathrm{red}}$ in $R$ be the ideal generated by these equations. 
Note that this ideal is reducible, as shown by the command \texttt{primaryDecomposition} in \texttt{Macaulay2}. One component consists of rank-deficient tuples $G^{ij'}$ and we call the other component $J_{\star}$. In particular, any tuple of rank-2 matrices $G^{ij'}$ satisfy \Cref{eq: compat aij} (up to scale) if and only if they satisfy the conditions of $J_{\star}$.

By \Cref{le: from Q to R}, if $G^{ij}$ are rank-2, on the form \Cref{eq: Gform}, and there exists $\lambda\neq 0$ with $\textbf G(\lambda)$ rank-4, then the entries of $G^{ij}$ satisfy the equations of $I_{\mathrm{rank}}$. In \texttt{Macaulay2} we see that the ideals $I_{\mathrm{rank}}$ and $J_{\star}$ are equal.  It follows that $G^{ij}$ satisfy the equations of $J_{\star}$. By the above, this implies that $G^{ij}$ are compatible, showing that the eigenvalue condition was not needed for compatibility. 

For $n=3$ in the non-collinear setting, we choose a fundamental action \begin{align}\begin{aligned}
    H_1&=\begin{bmatrix}
    e_1^2\,\, e_1^3\,\, \textbf{x}_1
    \end{bmatrix}, H_2=\begin{bmatrix}
    e_2^1\,\, e_2^3\,\, \textbf{x}_2
    \end{bmatrix}, H_3&=\begin{bmatrix}
    e_3^1\,\, e_3^2\,\, \textbf{x}_3
    \end{bmatrix},   
\end{aligned}
\end{align}
for $\textbf{x}_i$ making $H_i$ full-rank. Using this as our fundamental action, we get a new sextuple of fundamental matrices
\begin{align}G^{ij}=H_i^TF^{ij}H_j.
\end{align}
The sextuple $\Set{G^{ij}}$ is compatible if and only if $\Set{F^{ij}}$ is. Note that the epipoles of $G^{ij}$, denoted by $h_j^i$, are:
\begin{align}\begin{aligned}
    &{h}_1^2=[1,0,0], & &{h}_2^1=[1,0,0], & &{h}_3^1=[1,0,0],  \\  
    &{h}_1^3=[0,1,0], & &{h}_2^3=[0,1,0], & &{h}_3^2=[0,1,0].        
\end{aligned}
\end{align} 
The three matrices must be on the form: 
\begin{align}
\begin{aligned}\label{eq: G123 3noncol}
    &G^{12}=\left[\begin{matrix} 0&0&0\\0&x_{12}&y_{12}\\0&z_{12}&w_{12} \end{matrix}\right], &      &G^{13}=\left[\begin{matrix} 0&x_{13}&y_{13}\\0&0&0\\0&y_{13}&z_{13} \end{matrix}\right], &
    &G^{23}=\left[\begin{matrix} x_{23}&0&y_{23}\\0&0&0\\z_{23}&0&w_{23} \end{matrix}\right].
\end{aligned}
\end{align}
We work in the polynomial ring $R=\QQ[x_{ij},y_{ij},z_{ij},w_{ij},\lambda]$, where $1\le i<j\le 3$ consider the following $3$-view matrix:
\begin{align}\label{eq: 3view 76}
    \textbf G(\lambda):=\begin{bmatrix}
        0 & G^{12} & G^{13} \\
        G^{21} & 0 & \lambda G^{23}\\
        G^{31} & \lambda G^{32} & 0
    \end{bmatrix}.
\end{align}
The corresponding $I_{\mathrm{rank}}$, defined analogously to the collinear case, equals $\langle x_{12},x_{13},x_{23}\rangle$. This means $\textbf G(\lambda)$ being rank-6 for a $\lambda\neq 0$ implies $x_{12}=0,x_{13}=0,x_{23}=0$. 

As in the proof of \Cref{thm: 4tuple-condition}, if there is a solution of cameras $P_i$ with non-collinear centers to \Cref{eq: G123 3noncol}, then we may choose them to be
\begin{align}\begin{aligned}\label{eq:P123 77}
    &P_1=\left[\begin{matrix} 0&\alpha_1^1&0&*\\0&0&\alpha_1^2&*\\0&0&0&\alpha_1^3 \end{matrix}\right], 
    & &P_2=\left[\begin{matrix} \alpha_2^1&0&0&*\\0&0&\alpha_2^2&*\\0&0&0&\alpha_2^3\end{matrix}\right], 
    & &P_3=\left[\begin{matrix} \alpha_3^1&0&0&*\\0&\alpha_3^2&0&*\\0&0&0&\alpha_3^3 \end{matrix}\right],
\end{aligned}
\end{align}
where $\alpha_i^j$ are non-zero scalars, and $*$ are some other (possible zero) scalars. Computing the fundamental matrices of these four cameras, one can check that by the degrees of freedom of the cameras established in \Cref{eq:P123 77}, any triple of fundamental matrices on the form \Cref{eq: G123 3noncol} with $x_{ij}=0$ has a solution with cameras on the form \Cref{eq:P123 77}. It follows that if there is a non-zero scalar $\lambda$ for which \Cref{eq: 3view 76} is rank-6, then the triples of fundamental matrices $G^{ij}$ are compatible, which is sufficient.    

In the setting of $n=4$ in Case 1, we use the same ideas and therefore only sketch the proofs. Start with a 4-view matrix $\textbf F$ that is rank-6 and with block rows and columns of rank-3 as in \Cref{thm: KGGB}. Then take any sub 3-view matrix $\textbf F'$. It is at most rank-6. However, since the epipoles in each image are all distinct, all its block rows and columns must be rank-3. This is only possible if $\textbf F'$ is at least rank-6. Now we can apply the above to see that the three fundamental matrices of this 3-view matrix are compatible. In other words, we have triple-wise compatibility. Then we can assume the fundamental matrices to be of the form \Cref{eq: Gform c1} and look at the ideal generated by the $7\times 7$ minors given such matrices with indeterminate entries. Here we scale $G^{23},G^{24},G^{34}$ with $\lambda_1,\lambda_2,\lambda_3$, respectively. After saturation of $\lambda_i$ and rank-deficienly loci, and after elimination of $\lambda_i$, we get in each case an ideal that we call $I_{\mathrm{rank}}$. This ideal in each case describes the same conditions as the ideal generated by \Cref{eq:4_tuple_as_x_y}. This means that the rank condition implies compatibility.

Now we move on to general values of $n$. First, in the general collinear case, let $F^{ij}$ be fundamental matrices for which there are scalars $\lambda_{ij}$ such that the $n$-view matrix $\textbf F=(\lambda_{ij}F^{ij})$ is rank-4 and whose $3\times 3n$ and $3n\times 3$ block rows and columns are rank-2. By \Cref{thm:compatible_if_each_sextuple_is_compatible}, it suffices to show triple-wise compatibility. Take any 3-view submatrix $\textbf F'$. It is at most rank-4 and its block rows and columns at most rank-2. But since the fundamental matrices are rank-2, the block rows and columns must be at least rank-2 and it follows that the 3-view matrix itself is at least rank-4. Therefore triple-wise compatibility follows from an earlier step of this proof. By similar logic, if the $n$-view matrix $\textbf F$ instead is rank-6 with block rows and columns of rank 3, then this also applies for any sub 4-view matrix $\textbf F'$, since we assumed that any three epipoles in each image do not lie on a line. In particular, we are then in Case 1 and by the above, we have quadruple-wise compatibility. By \Cref{thm:compatible_if_each_sextuple_is_compatible}, this suffices.
\end{proof}

\section{The Cycle Theorem}\label{s: Can} 

Although the focus of this paper has been on complete graphs, in this section we state the cycle theorem, which holds for all graphs. We use this theorem to give an alternative derivation of necessary conditions for compatibility from \Cref{s: Kn}. We consider sets of fundamental matrices $\{F^{ij}\}$, where the index pairs $(ij)$ are a subset of all $n\choose 2$ possible ones. Let $\mathcal G=(V,E)$ denote the corresponding graph, where $V$ is the set of indices and $E$ the set of pairs of indices for which there is a fundamental matrix in our set. The definitions of compatibility and solution extend naturally to this setting. 

The theorem below gives a necessary and sufficient condition for when a set of fundamental matrices are compatible using the cycle condition for any graph $\mathcal G$. Recall that a directed cycle $C$ of a graph is a closed path, i.e. a path that starts and ends at the same vertex. Let $E(C)$ denote its directed edges. 

\begin{theorem}\label{thm:cycle} Let $\Set{F^{ij}}$ be a set of fundamental matrices with corresponding graph $\mathcal G$. $\Set{F^{ij}}$ is compatible if and only if there are matrices $H_i\in \mathrm{GL}_{3}$ and scalars $\lambda_{ij}=\lambda_{ji}\neq 0$ such that $G^{ij}:=\lambda_{ij}H_i^TF^{ij}H_j$ satisfy
\begin{align}\label{eq: cycleC}
    \sum_{(ij)\in E(C)}G^{ij}=0, \textnormal{ for each directed cycle }C \textnormal{ of } \mathcal G. 
\end{align}
In particular, any set of $3\times 3$ rank-2 matrices $G^{ij}$ satisfying the \emph{\textnormal{cycle condition}} \Cref{eq: cycleC} are the fundamental matrices of some set of cameras.
\end{theorem}

This theorem is very similar to the result \cite[Proposition 5]{arrigoni2018bearing}, which appears in the context of parallel rigidity and is relevant for the solvability of essential matrices. Observe that the cycles of length two in \Cref{eq: cycleC} imply that $G^{ij}$ are skew-symmetric.

\begin{proof}[Proof of \Cref{thm:cycle}, direction $\Rightarrow$] Let $\Set{F^{ij}}$ be a compatible set of fundamental matrices, with a solution of cameras $P_i$. By right action of $H\in \mathrm{GL}_4$, we may assume that the centers of these cameras has a non-zero last coordinate. Then the first three vectors must be linearly independent and the cameras can be written $[H_i|v^{(i)}]$, where $H_i\in \mathrm{GL}_3$ and $v^{(i)}\in \RR^3$. By left multiplication with $H_i^{-1}$, we may further assume that all cameras are of the form $C_i=[I| t^{(i)}]$, where $t^{(i)}\in \RR^3$. Recall that definition of $[t]_\times$ for $t\in \RR^3$ from \Cref{ss: K3}. One can check that the fundamental matrix of $C_i$ and $C_j$ is
\begin{align}\label{eq: fundcan}
\begin{aligned}
    [t^{(j)}]_\times-[t^{(i)}]_\times=[t^{(j)}-t^{(i)}]_\times \in \RR^{3\times 3},
\end{aligned}
\end{align}
and we call these skew-symmetric matrices $G^{ij}$. Note that $G^{ij}$ are scalings of $H_i^TF^{ij}H_j$. If we sum $G^{ij}$ for $(ij)$ in a cycle $C$, we must get 0 by \Cref{eq: fundcan}. 
\end{proof}

For the other direction, we need a lemma.

\begin{lemma}\label{le: graph} Let $\mathcal G$ be a connected graph and $T$ any spanning tree subgraph. Then there is a sequence $T^i\subseteq \mathcal G$ such that
\begin{align}
    T=T^0\subseteq \cdots \subseteq T^k=\mathcal G,
\end{align}
where $T^{i+1}$ contains exactly one more edge than $T^i$ and this edge is part of a cycle of $T^{i+1}$.
\end{lemma}

\begin{proof} We get $T^{k-1}$ from $T^k$ by removing an edge of $T^k$ that is not in $T$. We repeat this process until we reach $T^0$. To see that this suffices, assume by contradiction that the edge removed from $T^{i+1}$ is not part of a cycle of $T^{i+1}$. Then $T^i$ would have to be disconnected. This implies that $T$ cannot be connected, which is a contradiction.
\end{proof}

\begin{proof}[Proof of \Cref{thm:cycle}, direction $\Leftarrow$] 
We find a set of cameras $C_i$ such that $\psi(C_i,C_j)$ equals $G^{ij}$ for every edge of $\mathcal G$. Since $F^{ij}$ and $G^{ij}$ are equivalent under fundamental action, this is enough. We may without restriction assume that $\mathcal G$ is connected with $n$ nodes. Since $G^{ij}$ are skew-symmetric and rank-2, there are non-zero $g^{ij}\in \RR^3$ such that $G^{ij}=[g^{ij}]_\times$. The cycle condition is then equivalent to 
\begin{align}
    \sum_{(ij)\in E(C)}g^{ij}=0, \textnormal{ for each directed cycle }C \textnormal{ of } \mathcal G.
\end{align}
Let $T$ be a spanning tree subgraph of $\mathcal G$. 

Fix $i=1$ and let $t^{(1)}=0\in \RR^3$. To any node $v$ in $T$, there is a unique path with no repeated vertices from $1$ to $v$ in $T$, since $T$ is a tree. Let $\sigma_{u,v}=\{u=i_1,i_2,\ldots,i_k=v\}$ denote the unique path between two vertices $u,v$ of $T$. For $i>1$, define
\begin{align}\label{eq: tv sum}
  t^{(v)}:=\sum_{(ij)\in \sigma_{1,v} }g^{ij}.  
\end{align}
This gives us cameras $C_i=[I|t^{(i)}]$ for each $i=1,\ldots,m$. We must check that $G^{ij}=\psi(C_i,C_j)$ for every edge of $\mathcal G$. Recall that for cameras on this form, $\psi(C_i,C_j)=[t^{(j)}-t^{(i)}]_\times$. If $(ij)$ is an edge of $T$, then $t^{(j)}-t^{(i)}=g^{ij}$ by construction, which shows $G^{ij}=\psi(C_i,C_j)$. For $(ij)$ that are not edges of $T$, we proceed as follows. Consider the sequence $T^i$ of \Cref{le: graph}. We proceed via induction to show that $G^{ij}=\psi(C_i,C_j)$ for every edge of $T^l$ for any $l$. The base case $T^0=T$ is already done. Assume that $C_i$ satisfy $G^{ij}=\psi(C_i,C_j)$ for all edges of $T^l$. In $T^{l+1}$, there is precisely one new edge $(ij)$ and that edge is part of a cycle $C$ of $T^{l+1}$. Using \Cref{eq: tv sum}, we get after some cancellation for some vertex $u$ of the cycle that 
\begin{align}
  \psi(C_i,C_j)&=[t^{(j)}-t^{(i)}]_\times \\
  &=\sum_{(st)\in \sigma_{u,j}} [g^{st}]_\times-\sum_{(st)\in \sigma_{u,i}} [g^{st}]_\times.  
\end{align}
Since $G^{ij}$ are skew-symmetric by the conditions of the 2-cycles, $g^{ji}=-g^{ij}$. Therefore we get 
\begin{align}
\psi(C_i,C_j)=\sum_{(st)\in \sigma_{i,j}} [g^{st}]_\times.  
\end{align}
However, by the cycle condition for the cycle $C$, this equals $[g^{ij}]_\times$, which shows $G^{ij}=\psi(C_i,C_j)$ for every edge in $T^{l+1}$ and completes the induction.
\end{proof}


For the rest of this section, we apply the cycle theorem to find conditions that must hold for compatible fundamental matrices. For instance, let $G^{12},G^{13}$ and $G^{23}$ be fundamental matrices satisfying the cycle condition. By the 2-cycles, we can write $G^{ij}=[g^{ij}]_\times$ for some $g^{ij}\in \RR^3$. Letting $\{h_j^i\}$ be the epipoles of $\Set{G^{ij}}$ defined as
\begin{align}\label{eq: Gepip}
    h_j^i:=(g_1^{ij},g_2^{ij}, g_3^{ij})^T,
\end{align}
one can check that 
\begin{align}    (h_2^1)^TG^{23}h_3^1=\det [g^{12}\;g^{23}\;g^{31}].
\end{align}
Therefore, $g^{12}+g^{23}+g^{31}=0$ implies $(h_2^1)^TG^{23}h_3^1=0$. Now if $F^{12},F^{13}$ and $F^{23}$ are compatible fundamental matrices, then by the cycle theorem there is a scaling and fundamental action such that $G^{ij}=\lambda_{ij}H_i^TF^{ij}H_j$ satisfy the cycle condition. This means that $F^{ij}$ must satisfy $e_2^1F^{23}e_3^1=0$, hence giving us \Cref{eq_compatible}.

We next sketch an argument for why the $n$-view fundamental matrix $\mathbf{F}$ (see \Cref{s: Kn}) for compatible $\Set{F^{ij}}$ is at most rank 6 given appropiate scalings. For the sake of simplicity assume $n=4$, but note that the below principle directly extends to any $n$. Let $\Set{G^{ij}}$ be six fundamental matrices satisfying the cycle condition. Consider the $4$-view matrix $\textbf{G}=(G^{ij})_{ij}$. Subtracting the first row of $\textbf{G}$ from the other rows, we have 
\begin{align}
    \textbf{G}=&\left[\begin{matrix}0 & G^{12} & G^{13} & G^{14}\\  
G^{21} & 0 & G^{23} & G^{24}\\ 
G^{31} & G^{32} & 0 & G^{34}\\ 
G^{41} & G^{42} & G^{43} & 0\end{matrix}\right]
\sim \left[\begin{matrix}0 & G^{12} & G^{13} & G^{14}\\  
-G^{12} & -G^{12} & -G^{12} & -G^{12}\\ 
-G^{13} & -G^{13} & -G^{13} & -G^{13}\\ 
-G^{14} & -G^{14} & -G^{14} & -G^{14} \end{matrix}\right],\label{eq: Gauss}
\end{align}
where $\sim$ denotes equivalence under Gaussian eliminiation. The rank of the first three rows of \Cref{eq: Gauss} is at most 3, and the rank of the last nine rows is the rank of the first three columns of \Cref{eq: Gauss}, which is at most 3. In total, the matrix is of rank at most 6. Now if $\Set{F^{ij}}$ is a set of compatible fundamental matrices, there is a scaling and fundamental action such that $G^{ij}=\lambda_{ij}H_i^TF^{ij}H_j$ satisfy the cycle condition. Define the $n$-view fundamental matrix $\textbf F=(\lambda_{ij}F^{ij})_{ij}$. Since the rank of a matrix is invariant under conjugation, the above shows that $\mathrm{rank}\;\textbf F\le 6$.

Finally, we use the cycle theorem to give alternative proof that \Cref{eq: 4-tuple} is necessary to ensure compatibility. Let $\Set{G^{ij}}$ be 6 skew-symmetric matrices. Again, write $G^{ij}=[g^{ij}]_\times$ and let $\lambda_{ij}=\lambda_{ji}\neq 0$ be scalars such that $\lambda_{ij}G^{ij}$ satisfy the cycle condition. Then 
\begin{align}
    \lambda_{kl}g^{kl}=-\lambda_{jk}g^{jk}-\lambda_{ij}g^{ij}-\lambda_{li}g^{li},
\end{align}
for all indices $i,j,k,l\in \{1,2,3,4\}$ and it follows that 
\begin{align}
\begin{aligned}
    &\det [\lambda_{ij}g^{ij}\;\lambda_{jk}g^{jk}\;\lambda_{kl}g^{kl}]\\
    =&\det [\lambda_{ij}g^{ij}\;\lambda_{jk}g^{jk}\;-\lambda_{li}g^{li}]\\
    =&-\det[\lambda_{li}g^{li}\;\lambda_{ij}g^{ij}\;\lambda_{jk}g^{jk}].\label{eq: det3x3}
\end{aligned}
\end{align}
Factoring out the constants, and with $h_j^i$ defined as in \Cref{eq: Gepip}, we get
\begin{align}
    \lambda_{ij}\lambda_{jk}\lambda_{ki}(h_j^i)^TG^{jk}h_k^l=-\lambda_{li}\lambda_{ij}\lambda_{jk}(h_i^l)^TG^{ij}h_j^k.
\end{align}
Assuming that all epipolar numbers $ (h_j^i)^TG^{jk}h_k^l$ are non-zero, and recalling that $\lambda_{ij}$ are non-zero, we find 
\begin{align}\label{eq: epip ratio}
\frac{(h_j^i)^TG^{jk}h_k^l}{(h_i^l)^TG^{ij}h_j^k}=-\frac{\lambda_{li}}{\lambda_{ki}}.
\end{align}
Further, using $\lambda_{ij}=\lambda_{ji}$,
\begin{align}\label{eq: lam prod}
\frac{\lambda_{31}}{\lambda_{21}}\frac{\lambda_{12}}{\lambda_{32}}\frac{\lambda_{23}}{\lambda_{43}}\frac{\lambda_{34}}{\lambda_{24}}\frac{\lambda_{24}}{\lambda_{14}}\frac{\lambda_{41}}{\lambda_{31}}=1.
\end{align}
Combining \Cref{eq: epip ratio,eq: lam prod}, we get \Cref{eq: 4-tuple} for $\Set{\lambda_{ij}G^{ij}}$. Now if we start with a set of six compatible fundamental matrices $\Set{F^{ij}}$, then by \Cref{thm:cycle}, there is a fundamental action such that $G^{ij}=H_i^TF^{ij}H_j$ are skew-symmetric and there are scalars $\lambda_{ij}$ making the cycle condition hold for $\Set{\lambda_{ij}G^{ij}}$. Then \Cref{eq: 4-tuple} holds for $\Set{G^{ij}}$ and by the invariance of the epipolar numbers under fundamental action, we get \Cref{eq: 4-tuple} for $\Set{F^{ij}}$.


\section{Image of the Fundamental Map}\label{s: Image} Related to the study of the constraints satisfied by compatible fundamental matrices, is the image of the \emph{fundamental map} given a viewing graph $\mathcal{G}=(V,E)$:
\begin{align}
    \Psi_{\mathcal{G}}:(\PP^{3\times 4})^m&\dashrightarrow (\PP^{3\times 3})^{E},\\
    (P_1,\ldots,P_m)&\mapsto (\psi(P_i,P_j))_{(ij)\in E}.
\end{align}
The fundamental map sends real projective camera matrices to a set of corresponding fundamental matrices. We define the \emph{viewing graph variety} $\mathcal{V}_\mathcal{G}$ to be the \emph{Zariski closure} of the image $\mathrm{Im}\Psi_\mathcal{G}$. By Chevalley's theorem, in this case the Zariski closure is equal to the Euclidean closure \cite[Theorem 4.19]{michalek2021invitation}. 

A natural question from the algebraic geometry point of view is if this variety is described by the constraints we proposed in \Cref{s: Kn}, in the complete graph case. We prove that that is not the case, and leave it is as an open problem to describe the viewing graph variety precisely.

\begin{proposition} The viewing graph variety of $K_n$ for $n\ge 3$ is a proper subset of the variety in $(\PP^{3\times 3})^{n\choose 2}$ defined by the $3{n\choose 3}$ triple-wise constraints and the ${n\choose 4}$ quadruple-wise constraints of \Cref{thm: 4tuple-condition}.
\end{proposition}

Note that strictly speaking, $e_j^i$ is not a polynomial in the entries of $F^{ij}$, because there is no way to write a generator of
the left kernel of a matrix $X$ as a polynomial expression that works for every $3\times 3$ matrix of rank-2. However, one can for instance turn the expression $(e_i^s)^TF^{ij}e_j^s=0$ into a polynomial system in $F^{si},F^{ij}$ and $F^{sj}$ by defining the epipoles on affine patches of the fundamental matrices, which we don't explain here in further detail. In any case, for a rank-1 $3\times 3$ matrices, the epipoles are understood as the $0$ vector.

\begin{proof} We do the proof for $K_3$, but note that our counterexample below can be directly extended to any $K_n$. 

The Euclidean closure of the set of three camera matrices $(P_1,P_2,P_3)\in (\PP^{3\times 4})^3$ of different centers is all of $(\PP^{3\times 4})^3$. Then since $\mathcal{V}_{K_3}$ is the Euclidean closure of $\mathrm{Im}\Psi_{K_3}$, any of its elements can be arbitrarily approximated by the image of full-rank cameras. We give an example showing that the triple-wise constraints are not enough to describe $\mathcal{V}_{K_3}$ by finding an element that cannot be approximated in the way described above. Consider the following example:
    \begin{align}
        F^{12}=\begin{bmatrix}
            0 & 1 & 0 \\
            -1 & 0 & 0\\
            0 & 0 & 0
        \end{bmatrix},  F^{13}=\begin{bmatrix}
            0 & 0 & 1 \\
            0 & 0 & 0\\
            -1 & 0 & 0
        \end{bmatrix}.
    \end{align}
    We can assume that $P_1=[I|0]$ and $P_2=[I|(0;0;-1)]$. Then the following are the only options for $P_3$:
    \begin{align}
        P_3=\begin{bmatrix}
            1 & 0 & 0 & 0 \\
        a & 1+b & c & d\\
            0 & 0 & 1 & 0 
         \end{bmatrix},
    \end{align}
    for any $a,b,c,d$ such that $d\neq 0$. We get that     \begin{align}\label{eq: F23}
        F^{23}=\psi(P_2,P_3)=\begin{bmatrix}
            a & -1 & c+d \\
            b+1 & 0 & 0\\
            -d & 0 & 0
        \end{bmatrix}.
    \end{align}
    This matrix is rank-2 if and only if $d\neq 0$ or $b\neq -1$. Any such choice gives a triplet satisfying the triple-wise conditions. Also $F^{12},F^{13},S^{23}$ satisfy the triple-wise constraints, where
     \begin{align}
        S^{23}=\begin{bmatrix}
            0 & 0 & 0 \\
            0 & 0 & 0\\
            0 & 0 & 1
        \end{bmatrix},
    \end{align}
    because the epipole of a rank 1 matrix is 0.
    
    Now any arbitrarily small perturbance of $F^{12}$ and $F^{13}$ leads to an arbitrarily small change in the choice of $F^{23}$ from \Cref{eq: F23}. But no small perturbance of \Cref{eq: F23} equals $S^{23}$, which shows that $F^{12},F^{13},S^{23}$ does not lie in $\mathcal{V}_{K_3}$ and we are done. \end{proof}

\section{Conclusion} This paper provided explicit polynomial contraints as necessary and sufficient conditions for $n\choose 2$ fundamental matrices to be compatible. These polynomials were expressed in terms of the fundamental matrices and their epipoles, and are projectively well-defined, i.e. homogeneous. As a consequence of our work, the previously established necessary and sufficient condition \cite{kasten2019gpsfm} can be simplified by dropping the eigenvalue condition in certain cases. Our main tool was to define and use the fundamental action of sets of fundamental matrices. We gave a necessary and sufficient condition for compatibility that applied not only to complete graphs, but to any viewing graph. We used it to give an alternative derivation of necessary conditions for compatibility. In the final section, we introduced the viewing graph variety and gave a first result in the case of complete graphs.

{\small
\bibliographystyle{ieee_fullname}
\bibliography{VisionBib}
}
\end{document}